\documentclass[11pt,a4paper]{scrartcl}

\makeatletter
\DeclareOldFontCommand{\rm}{\normalfont\rmfamily}{\mathrm}
\DeclareOldFontCommand{\sf}{\normalfont\sffamily}{\mathsf}
\DeclareOldFontCommand{\tt}{\normalfont\ttfamily}{\mathtt}
\DeclareOldFontCommand{\bf}{\normalfont\bfseries}{\mathbf}
\DeclareOldFontCommand{\it}{\normalfont\itshape}{\mathit}
\DeclareOldFontCommand{\sl}{\normalfont\slshape}{\@nomath\sl}
\DeclareOldFontCommand{\sc}{\normalfont\scshape}{\@nomath\sc}
\makeatother

\usepackage{placeins}
\usepackage{float}
\usepackage{a4wide}
\usepackage{latexsym}
\usepackage[USenglish]{babel}
\usepackage[utf8]{inputenc}
\usepackage[T1]{fontenc}
\usepackage[breaklinks=true]{hyperref}
\usepackage{amssymb}
\usepackage{amsfonts}
\usepackage{amsmath,amssymb}
\usepackage[amsthm,thmmarks]{ntheorem}
\usepackage{hhline}
\usepackage{float}
\usepackage{mathrsfs}
\usepackage{graphicx}
\usepackage{amsmath}
\usepackage{mathdots}
\usepackage{amsfonts}
\usepackage{amsxtra}
\usepackage{url}
\usepackage{mathptmx}
\usepackage{helvet}
\usepackage{stmaryrd}
\usepackage{dsfont}
\usepackage{enumerate}
\usepackage{MathDef}
\usepackage{verbatim}
\usepackage{color}
\usepackage{cleveref}
\usepackage{cancel}
\usepackage[authoryear,round]{natbib}
\usepackage{scalerel,stackengine}
\usepackage{marginnote}

\stackMath
\newcommand\reallywidehat[1]{%
	\savestack{\tmpbox}{\stretchto{%
			\scaleto{%
				\scalerel*[\widthof{\ensuremath{#1}}]{\kern-.6pt\bigwedge\kern-.6pt}%
				{\rule[-\textheight/2]{1ex}{\textheight}}
			}{\textheight}%
		}{0.5ex}}%
	\stackon[1pt]{#1}{\tmpbox}%
}
\def\C{\mathbb{C}}
\def\E{\mathbb{E}}
\def\N{\mathbb{N}}
\def\P{\mathbb{P}}

\def\R{\mathbb{R}}
\def\Z{\mathbb{Z}}

\def\1{\mathds{1}}

\def\Var{\text{Var}}

\def\Cov{\operatorname{Cov}}

\def\f{f_Y^{(\Delta)}}
\def\tr{\operatorname{tr}}

\def\pa2{\frac{\partial ^2}{\partial \vartheta ^2}}

\newcommand\D{^{(\Delta)}}
\def\vecc{\operatorname{vec}}
\def\sc{\textsc}
\def\bf{\textbf}

\makeatother

\theoremstyle{plain}
\newtheorem{theorem}{Theorem}[section]
\newtheorem{lemma}[theorem]{Lemma}
\newtheorem{proposition}[theorem]{Proposition}

\newtheorem{corollary}[theorem]{Corollary}

\newtheorem*{assumptionletter2}{{\textbf{Assumption N}}}

\newtheorem{assumptionletter}{{\textbf{Assumption}}}
\theoremstyle{definition}
\newtheorem{example}[theorem]{Example}
\newtheorem{remark}[theorem]{Remark}

\newcommand{\VF}[1]{{\textcolor{blue}{#1}}}

\numberwithin{equation}{section}
\allowdisplaybreaks

\title{Empirical spectral processes for \\[2mm]  stationary state space models}

\author{Vicky Fasen-Hartmann \setcounter{footnote}{1}\thanks{Institute of Stochastics, Englerstra{\ss}e 2,
D-76131 Karlsruhe, Germany. \emph{Email:}
\href{mailto:vicky.fasen@kit.edu}{vicky.fasen@kit.edu}}
\and Celeste Mayer \setcounter{footnote}{0}\thanks{Institute of Stochastics, Englerstra{\ss}e 2,
D-76131 Karlsruhe, Germany. \emph{Email:}
\href{mailto:celestececilemayer@gmail.com}{celestececilemayer@gmail.com}}}

\date{}

\begin{document}
%
\maketitle

\begin{abstract}

   In this paper, we consider  function-indexed normalized  weighted integrated \linebreak periodograms for equidistantly sampled multivariate continuous-time state space models which are multivariate continuous-time ARMA processes. Thereby, the sampling distance is fixed and the driving  L\'evy process has at least a finite fourth moment. Under different assumptions on the function space and the moments of the driving Lévy process we derive a  central limit theorem for the  function-indexed normalized weighted integrated periodogram. Either the assumption on the function space or the assumption on the existence of moments of the Lévy process is weaker. Furthermore, we show the weak convergence in both the space of    continuous functions and in the dual space to a Gaussian process and give an explicit representation of the covariance function. The results can be used to derive the asymptotic behavior of the Whittle estimator
   and to construct goodness-of-fit test statistics as the Grenander-Rosenblatt statistic and the  Cram\'er-von Mises statistic.
   We present the exact limit distributions of both statistics and show
   their performance through a simulation study.
\end{abstract}

\noindent
\begin{tabbing}
\\ \> Secondary:  60G10, 62M86
\end{tabbing}

\vspace{0.2cm}\noindent\emph{Keywords:} Cramér-von Mises test, empirical spectral process, functional central limit theorem,    goodness of fit test,  Grenander-Rosenblatt test,  MCARMA process,  periodogram, state space model

	\section{Introduction}

In the context of stationary time series,
numerous estimators and testing procedures are based on the periodogram, the empirical version of the spectral density (see \cite{Brockwell:Davis:1991,Brillinger,Grenander:Rosenblatt,Priestley:1981}).
Typical examples for estimators  are the Whittle estimator for a parametric model and the empirical estimator of the spectral
distribution function. Classical goodness-of-fit tests for the spectral distribution function are the Grenander-Rosenblatt and
the Cram\'er-von Mises test statistic. They have in common that they have representations as functionals of empirical spectral
 processes
which are based on weighted integrated periodograms.
Therefore, to construct confidence bands for these estimators and asymptotic test statistics
the asymptotic behavior of the empirical spectral process is required.

In this paper, we derive the asymptotic behavior of the empirical spectral process for a  low-frequency sampled $m$-dimensional stationary
state space process $Y=(Y_t)_{t\geq 0}$ which is driven by a $d$-dimensional Lévy process $(L_t)_{t\geq 0}$. A
Lévy process $L=(L_t)_{t \geq 0}$ is a stochastic process with stationary and independent increments satisfying $L_0=0$ almost surely
and having continuous in probability sample paths.
Then a continuous-time linear state space model  with $A\in \R^{r\times r}$, $B\in \R^{r\times d}$, $C\in \R^{m\times r}$   is defined by
\begin{align}
\label{statespace}
\begin{array}{rcl}
dX_t&=& AX_tdt+ BdL_t, \\
Y_t&=& CX_t,   \quad t\geq 0.
\end{array}
\end{align}
In particular,
any multivariate continuous-time ARMA (MCARMA) process has a representation as a state space model (\cite{Marquardt:Stelzer:2007}). They are applied in
diversified fields as, e.g., in signal processing, systems and control, high frequency financial econometrics and financial mathematics.
In particular, a continuous-time state space model sampled discretely is a discrete-time state space model
with strong white noise and an ARMA process with a weak white noise (\cite{Thornton:Chambers:2017}).

In applications one often observes discrete data
although the data are coming from a continuous-time model. However, one is interested to know
the model parameters of the background continuous-time model because then the model parameters
for different sampling frequencies are known.
Therefore, in this paper, we observe the continuous-time process $Y$  at discrete-time
points with distance $\Delta>0$ and define
$Y^{(\Delta)}:=(Y_k^{(\Delta)})_{k\in \N_0}:=(Y_{k\Delta})_{k\in \N_0}$.
The empirical version of the spectral density
 $\f(\omega)=\frac{1}{2\pi}\sum_{h=-\infty}^\infty\Gamma_{Y^{(\Delta)}}(h)\e^{-ih\omega}$,
where $(\Gamma_{Y^{(\Delta)}}(h))_{h\in\Z}$ is the autocovariance function of $Y^{(\Delta)}$, is the periodogram
\begin{align*}
    I_{n,Y^{(\Delta)}}(\omega)=\frac{1}{2\pi}\sum_{h=-n+1}^{n-1}\overline{\Gamma}_{n,Y^{(\Delta)}}(h)e^{-ih\omega}
        =\frac{1}{2\pi n}\left(\sum_{j=1}^{n}Y_j\D e^{-ij\omega}\right)\left(\sum_{k=1}^{n}Y_k\D e^{ik\omega}\right)^\top, \quad  \omega \in [-\pi, \pi],
\end{align*}
where
$$\overline{\Gamma}_{n,Y^{(\Delta)}}(h):=\frac1n\sum_{k=1}^{n-h}Y_{k+h}\D Y_{k}^{(\Delta)\top} \quad \mbox{ and } \quad \overline{\Gamma}_{n,Y^{(\Delta)}}(-h):=\overline{\Gamma}_{n,Y^{(\Delta)}}(h)^\top, \quad h=0,\ldots,n,$$
is the sample autocovariance function.
Then, for a
 function $g: [-\pi,\pi]\to \C^{m\times m}$ with $\int_{-\pi}^{\pi}\|g(x)\|^2dx<\infty$ the normalized 
 weighted integrated periodogram is
\begin{align*}
E_n(g)
:=\sqrt n\int_{-\pi}^{\pi}g(\omega)\left(I_{n,Y^{(\Delta)}}(\omega)-\f(\omega)\right)d\omega.
\end{align*}
The topic of this paper is the asymptotic behavior of the empirical spectral process \linebreak $(\tr(E_n(g)))_{g\in\mathcal{G}_m}$ for some
function space $\mathcal{G}_m$ of square-integrable functions in $\mathcal{L}^2(\left[-\pi,\pi\right])$ where $\tr$ is the abbreviation for
trace. The application of the trace has the advantage that multiplication is getting commutative which is not given
in the multivariate setting. An important function space is $\mathcal{G}_m^F:=\{Id_m\1_{[0,t]}(\cdot):t\in[0,\pi]\}$ where $Id_m$ is the
$m\times m$ dimensional identity matrix. Indeed, the process
$(E_n(g))_{g\in\mathcal{G}_m^F}$ reflects the deviation of the empirical spectral distribution function $\int_0^t I_{n,Y^{(\Delta)}}(\omega)\,d\omega$ from the spectral
distribution function $\int_0^t\f(\omega)\,d\omega$. As we will see, and what is well-known for other time series models, the empirical spectral distribution
function is a
consistent estimator of the spectral distribution function although
the periodogram
is not a consistent estimator for the spectral density; see Theorem 3.1 in \cite{fasen2013a}. The Grenander-Rosenblatt statistic is then
\linebreak
$\sup_{t\in[0,\pi]}|\tr(E_n(Id_m\1_{[0,t]}(\cdot)))|$ and the Cramér-von Mises statistic is \linebreak
 $\int_0^\pi [\tr(E_n(Id_m\1_{[0,t]}(\cdot)))]^2\,dt$, respectively,
which are functionals of the empirical spectral process.
Often, one uses as well the self-normalized versions of them giving \linebreak $m^{-1/2}\sup_{t\in[0,\pi]}|\tr(E_n(\f(\cdot)^{-1}\1_{[0,t]}(\cdot)))|$
and   $m^{-1}\int_0^\pi [\tr(E_n(\f(\cdot)^{-1}\1_{[0,t]}(\cdot)))]^2\,dt$, respectively. In these cases the underlying function space of the empirical
spectral process is
$\mathcal{G}_m^S:=\{\f(\cdot)^{-1}\1_{[0,t]}(\cdot):t\in[0,\pi]\}$. As we show in this paper, the self-normalized versions
have the advantage that under the assumption that the driving Lévy process is a Brownian motion, the limit distribution
of $(\tr(E_n(g)))_{g\in\mathcal{G}_m^S}$ is not dependent on the model parameters anymore. The limit process is a time-scaled Brownian motion on $[0,\pi]$.
If the assumption that the Lévy process is a Brownian motion fails, the limit process has an additional Gaussian correction term depending
on a fourth order cumulant.
 A further popular example of a function space is $\mathcal{G}_m^\Gamma:=\{g_h:h\in\Z\}$ with $g_h(\omega)=\e^{ih\omega}$
which models the sample autocovariance function $(\int_{-\pi}^{\pi}g_h(\omega)I_{n,Y^{(\Delta)}}(\omega)d\omega)_{h\in\Z}=(\overline{\Gamma}_{n,Y^{(\Delta)}}(h))_{h\in\Z}$.
The last example we mention is the Whittle function which is the spectral analogue of the likelihood function in the time domain and its derivatives. They  have
 representations
as weighted integrated periodograms and thus, their limit behaviors can be derived via empirical spectral processes, see \Cref{Whittle} for further details.
These limit behaviors can
be used to prove  the asymptotic normality of the Whittle estimator.
From these examples we already see the broad applications of the empirical spectral process.
 Further examples are given in the overview
paper \cite{Dahlhaus:Poloik:2002}.

To the best of our knowledge there are not many papers studying the empirical spectral process for multivariate processes.
There is to mention \cite{dahlhaus1988empirical} who investigates the behavior of the  empirical spectral process for a wide class of multivariate  time series in discrete time
 with existing moments of all orders and a weak entropy condition for $\mathcal G_m$. However, the moment assumption,
 which is formally an assumption on the cumulant spectrum, is rather strong. The work  was extended  to univariate locally
stationary time series in \cite{dahlhaus2009empirical}. For univariate linear processes  \cite{Mikosch:Norvaisa}  consider the empirical
  spectral process under a stronger entropy condition on the function space $\mathcal G_m$  but  only assuming a finite fourth moment.
   Both papers, \cite{dahlhaus1988empirical} and  \cite{Mikosch:Norvaisa}, show the convergence of the empirical spectral process to a Gaussian process
   in the space of continuous functions on $\mathcal G_m$. But  the proof of \cite{Mikosch:Norvaisa} is not obvious, see \Cref{Remark_3_3}.
   In contrast, \cite{Bardet:Doukhan:Leon} shows the convergence of the empirical spectral process  for a wide class of weakly dependent discrete-time processes in the dual space of $\mathcal G_m$ under the assumption of a finite fourth moment  and a condition on $\mathcal{G}_m$ without using an entropy condition.  Empirical spectral processes for $\alpha$-stable linear processes are studied in \cite{Can::2010},
   however, the cases $\alpha\in(0,1)$ and $\alpha\in(1,2)$ are handled differently. In general, there is a conflict of goals having a weaker assumption
   on the function space $\mathcal G_m$ and a weaker moment assumption on the driving Lévy process. It seems challenging to get
   weak assumptions for both the function space $\mathcal G_m$ and the driving Lévy process.
   The special classes $\mathcal{G}_m^F$ and $\mathcal{G}_m^S$ resulting in the empirical spectral distribution function
   and the self-normalized empirical spectral distribution function, respectively, or the sample standardized empirical spectral distribution  function
   for different univariate time series models were examined in several papers.
   In particular, for short memory linear time series models with finite variance we refer to \cite{Anderson:93}, \cite{Grenander:Rosenblatt} and with infinite variance to \cite{Kluppelberg:1996}.
   \cite{Kokoszka:Mikosch:1997} cover both the finite and the infinite variance case for long memory univariate linear time series.
   A nice survey for linear time series models is \cite{Mikosch:1995}. However,  apart from \cite{dahlhaus1988empirical} these papers
   restrict to one-dimensional models and it seems that even multivariate ARMA models are not covered in the literature yet.
   But the results of the present paper hold as well for causal multivariate ARMA models driven by a strong white noise (see \Cref{Remark 3.3a}).

   In the above mentioned literature as well as in the present paper the proofs are based on the independence assumption
   of the white noise and it seems challenging to relax that assumption to receive results for causal multivariate ARMA models driven by a weak white noise. A reason is that without the independence assumption it is tricky to calculate higher moments.
    Unfortunately, a continuous-time state space model sampled discretely is only a multivariate ARMA process with a weak
    white noise and the exact representations of the ARMA parameters are not known. Therefore, it is difficult to use that approach for deriving the asymptotic behavior of the  continuous-time state space model sampled discretely.

The paper is structured on the following way. In \Cref{sec:Preliminaries}, we present preliminaries on discrete-time sampled state
space models and on the function spaces  $\mathcal{G}_m$ considered in this paper. The main achievements  are presented in \Cref{sec:main:results}:
The weak convergence of the empirical spectral process  to a Gaussian process in the space of continuous functions on $\mathcal{G}_m$ in \Cref{Satz::intP::Y2}
and in the dual space of $\mathcal{G}_m$ in \Cref{Satz::intP::Y} and \Cref{Theorem:main3}, respectively.
The  covariance function of the Gaussian process has an explicit representation given there. We distinguish different model assumptions allowing either weaker assumptions
on $\mathcal{G}_m$ or weaker moment assumptions on the driving Lévy process, respectively. The applications
of these results to construct goodness-of-fit tests for state space models are given in \Cref{sec:goodnessoffittest}. In particular,
 the Grenander-Rosenblatt and the Cram\'er-von Mises test statistics are further explored and their performance are demonstrated through a simulation study.
 Finally, \Cref{sec:Proofs} contains the proofs of the main theorems. The
 proofs of some auxiliary results  are moved to the Appendix.

	\subsection*{Notation}
	For some matrix $A$, $\tr(A)$ stands for the trace of $A$ and $A^\top$ for its transpose. The Kronecker product of some matrices $A$ and $B$ is denoted by $A\otimes B$. We write $A[S,T]$ for the $(S,T)$-th component of $A$ and $\| A\|=\sqrt{\sum_{S=1}^{r_1}\sum_{T=1}^{r_2}|A[S,T]|^2}=\sqrt{\tr(A^HA)}$ for the Frobenius norm of $A\in\C^{r_1\times r_2}$
where $A^H$ is the adjoint of $A$. The Frobenius norm can be replaced by any sub-multiplicative matrix norm with minor adaptions. The $r$-dimensional identity matrix is written as $Id_r$. Throughout the article, we write the $h$-th Fourier coefficient of some  square integrable function $g$ on $[-\pi,\pi]$  as
$\widehat g_h=\frac{1}{2\pi}\int_{-\pi}^\pi g(\omega)\e^{-i\omega h}\,d\omega \text{ for } h\in \Z$ such that
$g(\omega)=\sum_{h=-\infty}^{\infty}\widehat g_h\e^{i\omega h}$ for $\omega \in [-\pi,\pi]$. For convergence in distribution and convergence in probability, we write $\overset{\mathcal D}{\longrightarrow}$ and $\overset{\P}{\longrightarrow}$, respectively. We use $o_\P(1)$ and $O_\P(1)$ as shorthand for terms, which converge to 0 in probability and which are tight, respectively. 
	 Finally, $\mathfrak{C}>0$ is a constant which may change from line to line.

\section{Preliminaries} \label{sec:Preliminaries}

\subsection{State space models} \label{sec:statespacemodels}

For the  state space model \eqref{statespace} we assume the following conditions:

\begin{assumptionletter} $\mbox{}$
\begin{itemize}
    \item[(a)] $\E\|L_1\|^4<\infty$.
    \item[(b)] The eigenvalues of $A$ have strictly negative real parts and $CC^\top=Id_m$.
\end{itemize}
\end{assumptionletter}

Note that the Lévy process $(L_t)_{t\geq 0}$ can as well be extended  to the negative real numbers by
defining $L_t=L_t\1_{\{t\geq 0\}}-\widetilde L_{t-}\1_{\{t<0\}}$
where $(\widetilde L_t)_{t\geq 0}$ is an independent copy of the Lévy process $(L_t)_{t\geq 0}$.

\begin{remark} $\mbox{}$
\begin{itemize}
\item[(a)] Consequently, there exists a stationary causal version for the multivariate Ornstein-Uhlenbeck process $X$ as
$X_t=\int_{-\infty}^t \e^{A(t-u)}B\,dL_u$ for $t\geq 0$ (cf. \cite{Masuda:2004}) and hence,
a stationary causal
representation for $Y$ as $Y_t=\int_{-\infty}^t C\e^{A(t-u)}B\,dL_u$ for $t\geq 0$. For such a causal representation the
assumption that the matrix
$A$ has strictly negative real parts is necessary. In the following, we will always assume that $Y$ has such
a stationary causal representation.

\item[(b)] The assumption  $CC^\top=Id_m$ is not restrictive: The MCARMA processes as defined in \cite{Marquardt:Stelzer:2007}
are state space models and satisfy
this condition.
 Under the assumption
of finite second moments,  the classes of stationary linear state space models and
MCARMA models are equivalent (cf. \cite{Schlemm:Stelzer:2012}, Corollary 3.4).
Thus, for any state space model there exists a representation satisfying
$CC^\top=Id_m$.
\end{itemize}
\end{remark}

If  Assumption A holds, the discrete-time sampled
process $Y^{(\Delta)}$  satisfies the following conditions:

	\begin{proposition}[\cite{Schlemm:Stelzer:2012}, Theorem 3.6] \label{SSPsampledNat}~\\
	Let Assumption $A$ hold. Then 
	\begin{align*}
	Y_k^{(\Delta)}=CX_k^{(\Delta)} \quad \text{ and } \quad
	X^{(\Delta)}_k=e^{A\Delta}X^{(\Delta)}_{k-1}+N^{(\Delta)}_k,\quad k\in\N,
    \end{align*}
    where $$ N_k^{(\Delta)}=\int_{(k-1)\Delta}^{k\Delta}e^{A(k\Delta-u)}BdL_u, \quad k\in\Z.$$
	The sequence $(N_k^{(\Delta)})_{k\in \Z}$ in $\R^{r}$ is an i.i.d. sequence with mean zero and covariance matrix $\Sigma_N^{(\Delta)}=\int_0^{\Delta}e^{Au}B\Sigma_LB^\top e^{A^\top u}du.$
	Furthermore, $(Y_k^{(\Delta)})_{k\in \N_0}$ has the vector MA$(\infty)$ representation $$Y_k^{(\Delta)}=\sum_{j=0}^\infty \Phi_j N_{k-j}^{(\Delta)}, \quad k \in \N_0,$$ where  the MA polynomial is
\begin{align}\label{CoeffPhi}
    \Phi(x)=\sum_{j=0}^{\infty}\Phi_jx^j=\sum_{j=0}^{\infty}Ce^{A\Delta j}x^j\quad \text{ for }   x\in \C \text{ with }\|x\|=1.\end{align}
\end{proposition}
Due to Assumption $A$ the coefficients $\Phi_j$ in \eqref{CoeffPhi} are exponentially decreasing which directly implies \begin{align}\label{PhiBound}
\sum_{h=0}^\infty h^q \|\Phi_h\|<\infty \quad \text{for }q \in \N_0.
\end{align}
On basis of \cite{Brockwell:Davis:1991}, Theorem 11.8.3, the spectral density $\f$ of  $Y^{(\Delta)}$ is
\begin{eqnarray*} \label{spec}
        \f(\omega)&=&\frac{1}{2\pi}\Phi(e^{-i\omega})\Sigma_N^{(\Delta)}\Phi(e^{i\omega})^\top, 
        \quad \omega \in [-\pi,\pi]. \nonumber
\end{eqnarray*}

\subsection{Function spaces} \label{sec:functionspaces}

Next, we present the setup for the function spaces $\mathcal{G}_m$ in the definition of the empirical spectral process.
Therefore, define
\begin{align*}
    \mathcal H_{m}:=\{g: [-\pi,\pi]\to \C^{m\times m}\ |& \,\,\, \|g(\cdot)\|\in\mathcal{L}^2([-\pi,\pi]) \},
\end{align*}
and equip $\mathcal H_{m}$ with the norm
$$\|g\|^2_2:=\frac 1 {2\pi}\int_{-\pi}^{\pi}\|g(x)\|^2dx \quad \text{ for } g\in \mathcal H_{m},$$
which generates the metric $d_2(f,g)=\|f-g\|_2$ for $f,g \in \mathcal H_m$.
 Suppose $\mathcal G_{m}\subseteq \mathcal H_{m}$. Then, we define for $l\in\N$
the space
$
    \mathcal F_{m,l}:=\{F:\mathcal G_{m} \to \C^{l\times l}|\,\, F \text{ is bounded}\}
$
and equip $\mathcal F_{m,l}$ with the metric $d_{\mathcal G_{m}}$ generated by the norm
$$\|F\|_{\mathcal G_{m}}:=\sup_{g\in \mathcal G_m}|F(g)| \quad \text{ for }F\in\mathcal F_{m,l}.$$
Finally, we define
\beao
    \mathcal C(\mathcal G_m,\C^{l\times l}):=\{F\in\mathcal F_{m,l}| F \text{ is uniformly  continuous with respect to the metric }d_{\mathcal G_{m}}\}.
\eeao
The space $( \mathcal C(\mathcal G_m,\C^{l\times l}),d_{\mathcal G_{m}})$ is complete and hence, a Banach space.
 Note, a metric space $(\mathcal G, d)$ is  totally bounded iff its covering numbers $$N(\varepsilon,\mathcal G,d):=\inf\{u |\ \exists g_1,\ldots,g_u \in \mathcal G| \inf_{i=1,\ldots,u}d(g,g_i)\leq \varepsilon \ \forall g \in \mathcal G \}$$ are finite for every $\varepsilon>0$.
 If we assume additionally that $(\mathcal G_{m},d_{2})$
is totally bounded then $( \mathcal C(\mathcal G_m,\C^{l\times l}),d_{\mathcal G_{m}})$ is a separable Banach space.

For $g\in \mathcal G_{m}$ and $\Phi$ as in \eqref{CoeffPhi}  define the function
    \beao
        g^{\Phi}(\omega):=\Phi(e^{i\omega})^\top g(\omega) \Phi(e^{-i\omega}), \quad \omega\in[-\pi,\pi],
    \eeao
    with Fourier coefficients
    \beao
        \widehat{g^{\Phi}_h}=\frac{1}{2\pi}\int_{-\pi}^{\pi}\Phi(e^{i\omega})^\top g(\omega) \Phi(e^{-i\omega})e^{-ih\omega} d\omega, \quad h\in\Z.
    \eeao
Moreover, for $s\geq 0$ we define the space  $\mathcal H_m^s:=\{g\in \mathcal H_m:\|g\|_{\Phi,s}<\infty\}$ with  
    $$ \|g\|_{\Phi,s}^2:=\sum_{h=-\infty}^\infty(1+|h|)^{2s} \|\widehat{g^\Phi_h}\|^2.$$
Then $(\mathcal H_m^s,\|\cdot\|_{\Phi,s})$ is a normed space.
 Indeed, $\|g\|_{\Phi,s}=0$ implies $\|\widehat{g^\Phi_h}\|=0$ for all $h\in \Z$. A conclusion of \Cref{lemma:aux1} below is then $\|\widehat g_h\|=0$ for all $h\in\Z$
and hence, $g$ is zero almost everywhere.

\begin{remark} $\mbox{}$ \label{Remark 3.1}
\begin{itemize}
    \item[(a)] For $s=0$ we receive with Parseval's equality
$$ \|g\|_{\Phi,0}^2=\sum_{h=-\infty}^\infty \|\widehat{g^\Phi_h}\|^2
=\frac 1 {2\pi} \int_{-\pi}^\pi \|g^\Phi(\omega)\|^2d\omega= \|g^\Phi\|_{2}^2.$$
\item[(b)] Suppose $\sup_{g\in \mathcal G_m}\|g\|_2<\infty$. Then,
 $\sup_{\omega\in [-\pi,\pi]}\|\Phi(e^{i\omega})\|\leq \mathfrak C$  implies
 \begin{eqnarray*}
\sup_{g\in \mathcal G_m} \|g\|_{\Phi,0}^2=\sup_{g\in \mathcal G_m}\frac 1 {2\pi}\int_{-\pi}^\pi \left\| \Phi(e^{i\omega})^\top g(\omega)\Phi(e^{-i\omega})\right\|^2 d\omega \leq  \mathfrak C^2\sup_{g\in \mathcal G_m} \|g\|_2^2<\infty.
\end{eqnarray*}
The same arguments yield \begin{align}\label{FK_absch}
\sup_{g\in \mathcal G_{m}}\sum_{h=-\infty}^\infty\|\reallywidehat{\left(g(\cdot)\Phi(e^{i\cdot})\right)}_h\|^2<\infty.
\end{align}
\end{itemize}
\end{remark}
The next lemma relates  $\|\widehat g_h\|$ to $\|\widehat{ g^\Phi_h}\|$.

\begin{lemma} \label{lemma:aux1}
$\mbox{}$
\begin{itemize}
    \item[(a)] There exists a constant $\mathfrak C>0$ such that for any $g\in  \mathcal H_m$ and $h\in\N_0$:
        \beao
            \|\widehat g_h\|\leq \mathfrak C\left[\|\widehat{ g^\Phi_h}\|+\|\widehat{ g^\Phi_{h-1}}\|+\|\widehat{g^\Phi_{h+1}}\|\right].
        \eeao
    \item[(b)] Let $g\in  \mathcal H_m$. Suppose there exist some constants $\lambda,\mathfrak C_1>0$ such that
     $\|\widehat g_h\|\leq \mathfrak C_1\e^{-\lambda h}$ for any $h\in\N$. Then, there exist
     as well a constant $\mathfrak C_2>0$ and some $\nu>0$ such that  $\|\widehat {g^\Phi_h}\|\leq \mathfrak C_2\e^{-\nu h}$
     for any $h\in\N$.
\end{itemize}
 \end{lemma}
 A direct consequence of this lemma
 and the definition of $\|\cdot\|_{\Phi,s}$
 is the following:

\begin{corollary} \label{corollary:aux}
Let $\mathcal G_m \subseteq \mathcal H_m$ and $s\geq 0$. Suppose there exists a constant $\mathfrak C_1>0$
such that \linebreak $\sup_{g\in\mathcal G_m}\|g\|_{\Phi,s}\leq \mathfrak C_1$. Then, there exists a constant $\mathfrak C_2>0$
such that $\sup_{g\in\mathcal G_m}\|g\|_{2}\leq \mathfrak C_2$.
\end{corollary}

\section{The functional central limit theorem} \label{sec:main:results}

The function space $\mathcal G_m$  satisfies either of the following assumptions.

 \begin{assumptionletter} \label{AB} $\mbox{}$\\
 Let $\mathcal G_m\subseteq  \mathcal H_m$ be totally bounded, $h\in\mathcal{H}_m$  and $\|g\|_{\Phi,s}<\infty$ for any $g\in \mathcal G_m$ and
 some $s\geq 0$. Suppose that
 one of the following conditions hold:
 \begin{itemize}
 	\item[(B1)]   Suppose  $s>1/2$.
 	\item[(B2)] Suppose $s=0$ and there exists a constant $K_L>0$ such that the joint cumulant of $BL_1$  satisfies
    \beam \label{cum1}
        \text{cum}(BL_1[k_1],\ldots,BL_1[k_j]) \leq K_L^j \quad \text{ for all }  k_1,\ldots,k_j\in\{1,\ldots,r\} \text{ and }j\in \N,
    \eeam
    where $BL_1[k]$ denotes the $k$-th component of the random vector $BL_1$ in $\R^r$.
 Furthermore, assume that
    \beam \label{B3}
 \int_0^1 [\log(N(\varepsilon,\mathcal G_{m}, d_{2}))]^2d\varepsilon<\infty.
 \eeam
 	\item[(B3)]  Suppose $s=0$.   Let the support of $h$ be an interval and $h$  be continuously differentiable in the interior of its support. The function space is defined as
            $$
                \mathcal{G}_m:=\{h(\cdot)\1_{\left[-\pi,t\right]}(\cdot):t\in[-\pi,\pi]\}.$$
 \end{itemize}
\end{assumptionletter}
Indeed, \eqref{cum1} is already satisfied if there exists a constant $\widetilde K_L>0$ such that
 $$\text{cum}(L_1[k_1],\ldots,L_1[k_j]) \leq \widetilde K_L^j \quad \text{ for all }  k_1,\ldots,k_j\in\{1,\ldots,d\} \text{ and }j\in \N.$$
In the following, we present some examples for function spaces $\mathcal{G}_m$ satisfying \eqref{B3}.

\begin{example}\label{remark1}~
	\begin{itemize}
    \item[(a)]
         Let $\mathcal{G}_m$ be defined as in $(B3)$. Since $\sup_{x\in[-\pi,\pi]}\|h  (x)\|<\infty$, it is straightforward to see that the covering numbers satisfy
          $$N\left(\sup_{x\in[-\pi,\pi]}\|h  (x)\|/2n,\mathcal{G}_m,d_2\right)\leq n \quad \text{ for any } n\in\N.$$
         A direct consequence is that $N(\epsilon,\mathcal{G}_m,d_2)\leq \mathfrak{C}\epsilon^{-1}$
         for any $\epsilon>0$ and \eqref{B3} in (B2) is satisfied. In particular, this space is totally bounded. But in (B2)
          we have additionally the cumulant condition which is not necessary in (B3).  
         The function spaces in $(B3)$ do not satisfy $(B1)$. Due to \Cref{Remark 3.1}(b) and $h\in\mathcal{H}_m$, the condition
         $\sup_{g\in {\mathcal G}_m}\|g\|_{\Phi,0}=\sup_{t\in[0,\pi]}\|h(\cdot)\1_{\left[0,t\right]}(\cdot)\|_{\Phi,0}\leq\mathfrak{C}\|h\|_2<\infty$ is automatically satisfied.

    \item[(b)]Suppose $\mathcal{G}_1$ is a
      Vapnik-Chervonenkis class (VC class) with VC index $V(\mathcal{G}_{1})$, see \citet[Section 2.6.2]{vanderVaart:Wellner} for a definition, and $\widetilde g$ is an envelope with $\int_{-\pi}^\pi|\widetilde g(x)|^2\,dx<\infty$. Then, due to \citet[Theorem 2.6.7]{vanderVaart:Wellner}
      there exists a constant $\mathfrak{C}>0$ such that  $N(\varepsilon,\mathcal G_{1}, d_{2})\leq \mathfrak{C}\varepsilon^{-V(\mathcal{G}_{1})+1}$ for $0<\epsilon<1$ and hence, \eqref{B3} in (B2) is satisfied.
     \item[(c)] Let $(\Theta,\tau)$ be a compact metric space and
        $\mathcal{G}_m=\{g_\vartheta\in\mathcal{H}_m:\vartheta\in\Theta\}$. Suppose the map $\vartheta\mapsto g_{\vartheta}$
        is Hölder continuous with exponent $b>0$ and the necessary number of balls  to cover $\Theta$ of radius at most $\epsilon$
        is of order $\epsilon^{-a}$ for some $a>0$.  Then  
        \eqref{B3} in (B2) is satisfied because $N(\varepsilon,\mathcal G_{m}, d_{2})\leq \mathfrak{C}\epsilon^{-ab}$.
     \item[(d)] Suppose $\mathcal{G}_m$ is a finite-dimensional vector space and an integrable envelope exists.
        Then, $\mathcal{G}_m^{(ij)}:=\{g_{ij}:g\in\mathcal{G}_m\}$ forms  a finite-dimensional vector space of real
        functions. Then, due to \cite[Lemma II.28 and Lemma II.25]{Pollard:Conv} there exists
        a $w_{ij}>0 $ such that
         $N(\epsilon,\mathcal{G}_m^{(ij)},d_2)\leq \mathfrak{C}\epsilon^{-w_{ij}}$
        for any $\epsilon>0$. But this implies that
        the covering numbers $N(\epsilon,\mathcal{G}_m,d_2)$ behave polynomial which again results in \eqref{B3}.
\end{itemize}
\end{example}
We are able to present the main results of that paper.

\begin{theorem}\label{Satz::intP::Y2}
	Suppose Assumption $A$ and $B$ hold and $\sup_{g\in\mathcal G_m}\|g\|_{\Phi,s}<\infty$.
Then, as $n\to\infty$,
	\begin{align*}
        (\tr(E_n(g)))_{g\in \mathcal G_{m}}
            	\overset{\mathcal D}{\longrightarrow}
        (\tr(E(g)))_{g\in\mathcal G_{m}} \quad \text{in } (\mathcal C (\mathcal G_{m},\C),d_{\mathcal G_{m}}),
\end{align*}
where
\begin{align*}
\tr(E(g))=&\tr\left(W_0'\widehat{g^{\Phi}_0}\right)+\tr\left(\sum_{h=1}^{\infty}W_h[\widehat{g^{\Phi}_h}+{\widehat{g^{\Phi}_{-h}}}^\top] \right)
\end{align*}
and $W_0'$, $W_h$, $h\in\N$, are independent Gaussian random matrices with
\beam \label{def W}
    \begin{array}{rcl}
    \vecc(W_0')&\sim& \mathcal N(0, \E[N^{(\Delta)}_1N^{(\Delta)\top}_1 \otimes N^{(\Delta)}_1 N^{(\Delta)\top}_1]-\Sigma_N^{(\Delta)}\otimes \Sigma_N^{(\Delta)}) \quad \text{ and }\\
    \vecc(W_h)&\sim & \mathcal N(0, \Sigma_N^{(\Delta)}\otimes \Sigma_N^{(\Delta)}), \ h\in \N.
    \end{array}
\eeam
\end{theorem}
The basic idea is the following. Suppose $$I_{n,N^{(\Delta)}}(\omega)=\frac{1}{2\pi n}\left(\sum_{j=1}^{n}N_j\D e^{-ij\omega}\right)\left(\sum_{k=1}^{n}N_k\D e^{ik\omega}\right)^\top$$ is the periodogram  and $(E_{n,N^{(\Delta)}}(g^\Phi))_{g\in\mathcal{G}_m}$, respectively
 is the empirical spectral process of the
i.i.d. sequence
$(N_k^{(\Delta)})_{k\in\Z}$ such that
\beao
    \tr\left(E_{n,N^{(\Delta)}}(g^\Phi)\right)=\sqrt n \int_{-\pi}^{\pi}\tr\left(g^\Phi(\omega)\left(I_{n,N^{(\Delta)}}(\omega)-\frac{\Sigma_N^{(\Delta)}}{2\pi}\right)\right)\,d\omega.
\eeao
 Since under the trace it is allowed to commute the matrices we obtain
\beao
\tr(E_{n,N^{(\Delta)}}(g^\Phi))=\sqrt n \int_{-\pi}^{\pi}\tr\left(g(\omega)\left(\Phi(e^{-i\omega})I_{n,N^{(\Delta)}}(\omega)\Phi(e^{i\omega})^\top -\f(\omega)\right)\right)\,d\omega
\eeao
and hence, the representation
\beam \label{eq15}
    \tr\left(E_{n}(g)\right)=\tr(E_{n,N^{(\Delta)}}(g^\Phi))+E_{n,R}(g)
\eeam
with
\beao
    E_{n,R}(g):=\sqrt n \int_{-\pi}^{\pi}\tr\left(g(\omega)\left(I_{n,Y^{(\Delta)}}(\omega)-\Phi(e^{-i\omega})I_{n,N^{(\Delta)}}(\omega)\Phi(e^{i\omega})^\top \right)\right)  d\omega, \nonumber
\eeao
holds.
The term $\tr(E_{n,N^{(\Delta)}}(g^\Phi))$ is determining the asymptotic behavior of $\tr\left(E_{n}(g)\right)$, whereas
$E_{n,R}(g)$ is asymptotically neglectable. The details are given in \Cref{sec:Proofs}.

\begin{remark} \label{Remark 3.3a}
Suppose $\mathcal{G}_m\subseteq \mathcal H_m$ and $(\widetilde N_j)_{j\in\Z}$
is a strong white noise with finite fourth moment satisfying Assumption B, where in the cumulant condition (B2) the random vector $BL_1$
is replaced by $\widetilde N_1$.
Define then a discrete time MA process
of the form
\beao
    \widetilde Y_k=\sum_{j=0}^{\infty}\widetilde \Phi_j\widetilde N_{k-j}, \quad k\in\N,
\eeao
where $\sum_{j=1}^\infty j^2\|\widetilde \Phi_j\|<\infty$.
The proof of \Cref{Satz::intP::Y2} shows that the results of \Cref{Satz::intP::Y2} stay true for
this MA process where $\Phi_j$ is replaced by $\widetilde \Phi_j$ and $N_1^{(\Delta)}$ is replaced by
$\widetilde N_1$, respectively.
In particular, every causal multivariate ARMA process driven by a strong white noise
has such a representation. In summary, we have derived as well the asymptotic behavior of the empirical spectral
process of a causal multivariate ARMA model with a strong white noise. 
\end{remark}

\begin{remark} $\mbox{}$ \label{Remark_3_3}
\begin{itemize}
\item[(a)] In the case that the driving Lévy process is a Brownian motion it is possible to weaken the entropy condition in $(B2)$;
for further details see \cite{dahlhaus1988empirical}, Remark~2.6.
\item[(b)] \cite{Mikosch:Norvaisa} derive the asymptotic behavior of the empirical spectral process for  univariate linear processes. Thereby, they mainly assume a finite fourth moment of the white noise and an entropy condition which is stronger 
    than the entropy condition in $(B2)$. However, the proof of Lemma 5.4 in that paper is based on the assertion that the quadratic form $\widetilde Q_n^2(\widetilde Y^2)$  is uniformly bounded which is 	questionable.
\item[(c)]  \cite{dahlhaus2009empirical} present a functional central limit theorem for locally stationary univariate time series under the entropy condition $(B2)$ with the cumulant condition replaced by a moment condition. To be more precise for the one-dimensional white noise $(\xi_k)_{k\in\N}$ they assume
     that $\E|\xi_1|^k\leq  C_{\varepsilon}^k$ for any $k\in\N$ and some constant $C_{\varepsilon}>0$. This implies that the
     $k$-th cumulant $\text{cum}_k(\xi_1)\leq k!\, C_{c}^k$ for any $k\in\N$  and some constant $C_{c}>0$ (\cite{Saulis:Statulevicius}, Lemma 3.1). As the example of the uniform
     distribution shows,  this upper bound is strict and it is not possible to conclude $\text{cum}_k(\xi_1)\leq C_{c}^k$ for any $k\in\N$.
      Therefore, it is not obvious why in the proof of Lemma~5.7 in \cite{dahlhaus2009empirical}, on page 28, the upper bound   holds.
      Thus, we are not using directly Lemma~5.7 of that paper, although it would be much easier for us, and we are using the cumulant condition
      instead of the moment condition. 
\item[(d)] In the case of heavy tailed models, as for $\alpha$-stable ARMA models (cf. \cite{Can::2010}),
the term $W_0'$ vanishes. The reason is that in heavy tailed models the sample variance and the sample autocovariance
function have different convergence rates, and hence, the sample variance has no influence on the limit behavior.
However, in the case of light tailed models this is not the case, and we receive additionally a term depending on the
 fourth cumulant of the white noise.
\end{itemize}
\end{remark}
Next, we derive the autocovariance function of the limit Gaussian process.

\begin{corollary}\label{Satz::intP::Y1}
Let the assumptions of \Cref{Satz::intP::Y2} hold.
Furthermore, suppose that $\widehat{g^{\Phi}_0}={\widehat{g^{\Phi}_0}}{}^\top$ for any $g\in \mathcal{G}_m$.
\begin{itemize}
\item [(a)] Then, $(\tr(E(g)))_{g\in\mathcal G_{m}}$ is a centered Gaussian process with covariance function
\begin{eqnarray*}
\lefteqn{ \Cov(\tr(E(g_1)),\tr(E(g_2)))}\\&=&\pi\int_{-\pi}^\pi \tr\left(\f(\omega) (g_1(\omega)+g_1(-\omega)^\top)^H \f(\omega) (g_2(\omega)+g_2(-\omega)^\top)\right) d\omega \\
			&&+\vecc\left(\frac 1 {2\pi}\int_{-\pi}^\pi \Phi(e^{i\omega})^\top g_1(\omega)\Phi(e^{-i\omega})d\omega\right)^\top  \left(\E[N^{(\Delta)}_1N^{(\Delta)\top}_1 \otimes N^{(\Delta)}_1 N^{(\Delta)\top}_1]\right.\\
			&&\quad\quad\quad\left.-3\Sigma_N^{(\Delta)}\otimes \Sigma_N^{(\Delta)}\right)\vecc\left(\left(\frac 1 {2\pi}\int_{-\pi}^\pi \Phi(e^{i\omega})^\top g_2(\omega)\Phi(e^{-i\omega})d\omega\right)^H\right).
		\end{eqnarray*}
\item[(b)]	The representation
\begin{align*}
\tr(E(g))=&\tr\left(W_0^* \widehat{g^{\Phi}_0}\right)+\tr\left(\sum_{h=-\infty}^{\infty}\frac{W_h}{\sqrt{2}}[\widehat{g^{\Phi}_h}+{\widehat{g^{\Phi}_{-h}}}^\top] \right)
\end{align*}
holds, where $W_0^*$, $W_h$, $h\in\Z$, are independent Gaussian random matrices with
\beam \label{def W}
    \begin{array}{rcl}
    \vecc(W_0^*)&\sim& \mathcal N(0, \E[N^{(\Delta)}_1N^{(\Delta)\top}_1 \otimes N^{(\Delta)}_1 N^{(\Delta)\top}_1]-3\Sigma_N^{(\Delta)}\otimes \Sigma_N^{(\Delta)}) \quad \text{ and }\\
    \vecc(W_h)&\sim & \mathcal N(0, \Sigma_N^{(\Delta)}\otimes \Sigma_N^{(\Delta)}), \ h\in \Z.
    \end{array}
\eeam
\item[(c)] Suppose that the driving Lévy process is a $d$-dimensional Brownian motion. Then, $W_0^*=0_{r\times r}$ a.s. and
the covariance function reduces to
\begin{eqnarray*}\lefteqn{\Cov(\tr(E(g_1)),\tr(E(g_2)))}\\
			&=&\pi\int_{-\pi}^\pi \tr\left(\f(\omega) (g_1(\omega)^\top +g_1(-\omega)) \f(\omega) (g_2(\omega)+g_2(-\omega)^\top)\right) d\omega.
\end{eqnarray*}
\end{itemize}
\end{corollary}

For a Hilbert space $\mathcal G_m$ we are not only able to
show the convergence in the space $(\mathcal{C}( \mathcal G_m,\C),d_{\mathcal G_m})$ but also
we are able to derive an analog result in the dual space $\mathcal{G}_m^{'}=\{F:\mathcal G_{m}\to\C|\, F\text{ is linear}\}$ with the operator norm
$$\|F\|_{\mathcal G_m'}^{\Phi,s}=\sup_{g\in\mathcal G_m \atop {\|g\|_{\Phi,s}\leq 1}}|F(g)|\qquad \text{ for } F\in\mathcal G_m'.$$
The corresponding metric is denoted by $d_{\mathcal G_m'}^{\Phi,s}$. Of course, for $F\in \mathcal G_m' $
the upper bound $\|F\|_{\mathcal G_m}\leq \|F\|_{\mathcal G_m'}^{\Phi,s}$ holds.
Therefore, it is not surprising that we have weaker assumptions for the convergence in
$(\mathcal G_m',{d}_{\mathcal G_m'}^{\Phi,s})$ than in $(\mathcal{C}( \mathcal G_m,\C),d_{\mathcal G_m})$:
In the dual space $(\mathcal G_m',{d}_{\mathcal G_m'}^{\Phi,s})$, the assumption $\sup_{g\in\mathcal G_m}\|g\|_{\Phi,s}<\infty$ is not required. 

\begin{theorem}\label{Satz::intP::Y}
	Suppose Assumption $A$ and $B$ hold, and $\mathcal G_m$ is a Hilbert space.
    Then, as $n\to\infty$,
	\begin{align*}
        (\tr(E_n(g)))_{g\in \mathcal G_{m}}
            	\overset{\mathcal D}{\longrightarrow}
        (\tr(E(g)))_{g\in\mathcal G_{m}} \quad \text{in } (\mathcal G_{m}',{d}_{\mathcal G_m'}^{\Phi,s}),
        \end{align*}
where $(\tr(E(g)))_{g\in\mathcal G_{m}}$ is defined as in \Cref{Satz::intP::Y2}.
\end{theorem}

\begin{remark}
Note,  $\mathcal{G}_m=\{h(\cdot)\1_{\left[\pi,t\right]}(\cdot):t\in[-\pi,\pi]\}$ in $(B3)$ is not a Hilbert space and hence, this case is not covered in \Cref{Satz::intP::Y}.
\end{remark}

However, a more general result holds without requiring that $\mathcal G_m$ is totally bounded.

\begin{theorem} \label{Theorem:main3}
For some $s>1/2$, define the set $\mathcal H_m^s:=\{g\in \mathcal H_m:\|g\|_{\Phi,s}<\infty\}$. Then,
$$(\tr(E_n(g)))_{g\in \mathcal H_{m}^s}
            	\overset{\mathcal D}{\longrightarrow}
        (\tr(E(g)))_{g\in\mathcal H_{m}^s} \quad \text{in } ({\mathcal H_{m}^s}',{d}_{{\mathcal H_m^s}'}^{\Phi,s}). $$
\end{theorem}
Indeed, the space $\mathcal H_m^s$ is a Hilbert space with scalar product $$\langle g,f \rangle _{\Phi,s}=\sum_{h=-\infty}^\infty (1+|h|)^{2s}{\tr\left(\widehat {g_h^\Phi}\widehat {f_h^{\Phi}}^\top \right)}\text{ for } g,f \in \mathcal H_m^s.$$

\begin{example}\label{remark1}~
    Since the autocovariance functions of state space models are exponential decreasing, \Cref{lemma:aux1}(b)  implies  that the spectral densities
    of discretely observed $m$-dimensional state space models are in $\mathcal H_m^s$ and in particular, the spectral densities of
    $m$-dimensional ARMA processes are in $\mathcal H_m^s$.
    Similarly, the inverses of these spectral densities are in $\mathcal H_m^s$  since the spectral densities
    are rational matrix functions and hence, the inverses are rational matrix functions as well with exponential decreasing
    Fourier coefficients. The process $Y^{(\Delta)}$ is an $m$-dimensional ARMA process.
\end{example}

\begin{example} \label{Whittle}
The Whittle estimator is a well-known parameter estimator in the frequency domain  going back to
\cite{whittle1953estimation},
 and is well investigated for different time series models in discrete
time. The Whittle estimator
 $\widehat{\vartheta}_n^{(\Delta)}:= \arg \min_{\vartheta\in \Theta} W_n(\vartheta)$
is the minimizer of the Whittle function
 $$W_n(\vartheta)=\frac{1}{2\pi}\int_{-\pi}^\pi\Big[\tr\left(f_Y^{(\Delta)}(\omega, \vartheta)^{-1}I_n(\omega)\right)+\log\left(\det \left(f_Y^{(\Delta)}(\omega,\vartheta)\right)\right)\Big], \quad  \vartheta\in\Theta,$$
    where $f_Y^{(\Delta)}(\omega, \vartheta)$ is a spectral density for any parameter $\vartheta$ in the a parameter space $\Theta$.
 Defining
 $$W(\vartheta)=\frac{1}{2\pi}\int_{-\pi}^\pi\Big[\tr\left(f_Y^{(\Delta)}(\omega, \vartheta)^{-1}f_Y^{(\Delta)}(\omega)\right)+\log\left(\det \left(f_Y^{(\Delta)}(\omega,\vartheta)\right)\right)\Big], \quad  \vartheta\in\Theta,$$
 we receive due to \Cref{Theorem:main3} and \Cref{remark1}
 under some mild assumptions
 \beao
    \sup_{\vartheta\in\Theta}|W_n(\vartheta)-W(\vartheta)|=n^{-1/2}\sup_{\vartheta\in\Theta}|E_n(f_Y^{(\Delta)}(\omega, \vartheta)^{-1}/2\pi)|=o_p(1).
 \eeao
Further, it is also possible to derive the asymptotic normality of the Whittle estimator using the asymptotic behavior of the empirical
spectral process. The  basic ideas of such an approach
are given in \cite{Bardet:Doukhan:Leon}, \cite{dahlhaus1988empirical} and \cite{Dahlhaus:Poloik:2002}.
\cite{Fasen:Mayer:2020} prove that the Whittle estimator for state space models with finite fourth moment is a
consistent and asymptotically normally distributed estimator without using the empirical spectral process.
However, their Whittle function is defined
by a sum which approximates the above integral.
\end{example}

\section{Goodness-of-fit tests} \label{sec:goodnessoffittest}

\subsection{Theory}

In this section, we investigate the behavior of some goodness-of-fit test statistics which are based on the empirical spectral distribution function.

\begin{theorem}\label{GOF-stat}~\\
	Let Assumption $A$ hold and let $W_0^*$, $W_h$, $h\in\Z$, be independent Gaussian random matrices as defined
in \eqref{def W} and $(B_t)_{t\in[0,1]}$ be a one-dimensional Brownian motion.
Then, the following statements hold:
	\begin{itemize}
		\item[(a)]The Grenander-Rosenblatt statistic satisfies as $n\to\infty$,
\begin{eqnarray*}
\lefteqn{\sqrt n\sup_{t\in [0,\pi]}\left|\tr\left(\int_{0}^t I_{n,Y^{(\Delta)}}(\omega)-\frac 1 {2\pi} \Phi(e^{-i\omega})\Sigma_N^{(\Delta)}\Phi(e^{i\omega})^\top d\omega \right)\right|} \\
			&&\overset{\mathcal D}{\longrightarrow}\sup_{t\in [0,\pi]}\left|\tr\left(\frac{W_0^*}{2\pi}\int_{0}^t \Phi(e^{i\omega})^\top  \Phi(e^{-i\omega})d\omega\right)\right.
			\\&& \quad\quad \left.+\tr\left(\sum_{h=-\infty}^{\infty}\frac{W_h}{2\sqrt{2}\pi}\left(\int_{-t}^{t}\Phi(e^{i\omega})^\top \Phi(e^{-i\omega})e^{-ih\omega}d\omega \right)\right)\right|.
		\end{eqnarray*}
	If the driving Lévy process is a Brownian motion the limit process  reduces to $$ \sup_{t\in [0,\pi]}\left|\tr\left(\sum_{h=-\infty}^{\infty}\frac{ W_h}{2\sqrt{2}\pi}\left(\int_{-t}^{t}\Phi(e^{i\omega})^\top \Phi(e^{-i\omega})e^{-ih\omega}d\omega 
\right)\right)\right|.$$
		\item[(b)] The Cramér-von Mises statistic satisfies as $n\to\infty$,
\begin{eqnarray*}\lefteqn{ n \int_{0}^\pi \left[\tr\left(\int_{0}^t I_{n,Y^{(\Delta)}}(\omega)-\frac 1 {2\pi} \Phi(e^{-i\omega})\Sigma_N^{(\Delta)}\Phi(e^{i\omega})^\top d\omega \right)\right]^2dt} \\
			&&\overset{\mathcal D}{\longrightarrow}\int_{0}^\pi\left[\tr\left(\frac{W_0^*}{2\pi}\int_{0}^t \Phi(e^{i\omega})^\top  \Phi(e^{-i\omega})d\omega\right.\right.
			\\&&\quad\quad\left.\left. +\sum_{h=-\infty}^{\infty}\frac{W_h}{2\sqrt{2}\pi}\left(\int_{-t}^{t}\Phi(e^{i\omega})^\top \Phi(e^{-i\omega})e^{-ih\omega}d\omega
\right)\right)\right]^2dt.
		\end{eqnarray*}
		If the driving Lévy process is a Brownian motion the limit process  reduces to  $$ \int_{0}^\pi\left[\tr\left(\sum_{h=-\infty}^{\infty}\frac{ W_h}{2\sqrt{2}\pi}\left(\int_{-t}^{t}\Phi(e^{i\omega})^\top \Phi(e^{-i\omega})e^{-ih\omega}d\omega 
\right)\right)\right]^2dt.$$
				\item[(c)] The self-normalized  Grenander-Rosenblatt statistic
            satisfies as $n\to\infty$,
			\begin{eqnarray*}
            \lefteqn{\frac{\sqrt n}{\sqrt{m}}\sup_{t\in [0,\pi]}\left|\int_{0}^t\tr\left( I_{n,Y^{(\Delta)}}(\omega)f^{(\Delta)}_Y(\omega)^{-1}\right)d\omega -t m\right|} \\
					&&\overset{\mathcal D}{\longrightarrow}
        \sup_{t\in [0,\pi]}\left|\tr\left(\frac{W_0^*}{2\pi\sqrt{m}}\int_{-t}^t \Phi(e^{i\omega})^\top f_Y^{(\Delta)}(\omega)^{-1} \Phi(e^{-i\omega})d\omega\right) +\sqrt{2}\pi B_{\frac{t}{\pi}}\right|.
				\end{eqnarray*}
				If the driving Lévy process is a Brownian motion then the limit distribution is equal to \linebreak
                $\sqrt{2}\pi\sup_{t\in [0,1]}|B_t|,$
                where $\sup_{t\in [0,1]}|B_t|$ has distribution function
            \begin{align} \label{distributionF}
                    F(x)=\sum_{k=-\infty}^\infty(-1)^k\left[\Phi((2k+1)x)-\Phi((2k-1)x)\right], \quad x\geq 0,
            \end{align}
                see \cite{billingsley}, Equation (9.14).
		\item[(d)] The self-normalized Cramér-von Mises statistic satisfies as $n\to\infty$,
		\begin{eqnarray*}\lefteqn{ \frac{n}{m} \int_{0}^\pi \left[\int_{0}^t \tr\left(I_{n,Y^{(\Delta)}}(\omega)f^{(\Delta)}_Y(\omega)^{-1}\right)d\omega -t m \right]^2dt} \\
			&&\overset{\mathcal D}{\longrightarrow}\int_{0}^\pi\left[\tr\left(\frac{W_0^*}{2\pi\sqrt{m}}
            \int_{0}^t \Phi(e^{i\omega})^\top f_Y^{(\Delta)}(\omega)^{-1} \Phi(e^{-i\omega})d\omega\right) +\sqrt{2}\pi B_{\frac{t}{\pi}}\right]^2dt.
		\end{eqnarray*}
		If the driving Lévy  process is a Brownian motion the limit distribution reduces to
        $2\pi^3\int_{0}^1 B_{t}^2 dt$.
		\end{itemize}
\end{theorem}
\begin{proof} $\mbox{}$
(a) and (b): Define the set $\mathcal{G}_m:=\{Id_m\1_{[0,t]}(\cdot): t\in[0,\pi]\}$.
Due to  condition $(B3)$ we are allowed to apply
\Cref{Satz::intP::Y2} such that an application of the continuous mapping theorem results  in
the statements (a) and (b), respectively. \\
(c) and (d): 
Similar arguments as in (a) and (b) with $\mathcal{G}_m:=\{\f(\omega)^{-1}\1_{[0,t]}(\cdot): t\in[0,\pi]\}$ give the statement.
The Brownian motion is popping up because due to  \Cref{Satz::intP::Y1} the covariance function
of the stochastic process $$\left(\frac{1}{\sqrt{m}}\tr\left(\sum_{h=-\infty}^{\infty} \frac{W_h}{2\sqrt 2 \pi}\int_{-t}^{t}
    \Phi(e^{i\omega})^\top f_Y^{(\Delta)}(\omega)^{-1} \Phi(e^{-i\omega}) e^{i\omega h}d\omega\right)\right)_{t\in [0,\pi]}$$  is
    equal to $2\pi\min\{s,t\}$ for $s,t\in[0,\pi]$.
\end{proof}
\begin{remark}~
In the case of one-dimensional ARMA processes several spectral goodness-of-fit test statistics were already investigated; see Section 6.2.6  of \cite{Priestley:1981}. The limit distribution of the Grenander-Rosenblatt statistic for ARMA processes with normally distributed white noise is the same as ours (see \cite{Grenander:Rosenblatt}).
Indeed, \cite{Anderson:93} uses for linear processes with finite second moments
the sample standardized periodogram to estimate the standardized spectral distribution function and obtains as limit the Brownian bridge if the noise is Gaussian.
 In contrast, the convergence rate and the limit distribution for $\alpha$-stable ARMA processes differ,
the convergence rate is faster. 
Instead of the Brownian motion $(B_t)_{t\geq 0}$ an analogue to the Brownian bridge for stable models occurs;
for more details see \cite{Kluppelberg:1996}, Section 4 .
\end{remark}

\subsection{Simulations}

The simulation study has two major purposes. Firstly, we want to find out if the theoretical results can be observed for finite sample sizes. 
Therefore, we investigate the behavior of the empirical and limit quantiles of the spectral goodness-of-fit test statistics.
Subsequently, we use the quantiles of the limit  process to construct some tests. These tests will be applied in different scenarios.

In the following, we focus on the self-normalized versions of the Grenander-Rosenblatt and the Cram\'er-von Mises statistic, see \Cref{GOF-stat} (c) and (d). We start with investigating their performances in the case of a univariate CARMA(2,1) process 
with  \begin{align*}A&=\left( \begin{array}{r r} 0 & 1  \\ -1&-1 \end{array}\right),\quad \quad B=\left( \begin{array}{r}  -1\\ 0\end{array}\right),
\qquad C=\left( \begin{array}{r r r}1 & 0 \end{array}\right),\quad \quad  \quad \quad \Sigma_{L}=1,\end{align*} a bivariate Ornstein-Uhlenbeck process (MCAR(1) process) with  \begin{align*}A&=\left( \begin{array}{r r} -1/2 & -1/2  \\ 1&-1 \end{array}\right)= B,
\qquad C=Id_2=\Sigma_{L},\end{align*} and a bivariate CARMA(2,1) process with  \begin{align*}A&=\left( \begin{array}{r r r }-1 & 4&0  \\0&0&1\\-1&0&-3  \end{array}\right) , \qquad B=\left( \begin{array}{r r }-1 & 4  \\-1&-1\\2&3  \end{array}\right), \quad
C=\left( \begin{array}{r r r}1 & 0& 0 \\0&1&0  \end{array}\right), \quad \Sigma_L=Id_2 .\end{align*}
In all settings, the processes are simulated with an Euler-Maruyama scheme  with initial values $Y(0)=X(0)=0$, step size 0.01 and observation distance $\Delta=1$. We take the Brownian motion and the normal-inverse Gaussian process (NIG L\'evy process) as driving process. The NIG L\'evy process is often used in financial applications as in the modeling of stock returns or stochastic volatility, see \cite{barndorff1997normal}.
The estimation results for finite sample sizes are based on 5000 replicates, whereas the estimation results
 corresponding to the limit processes are based on $50.000$ replicates with the infinite series replaced by a sum consisting of $250$ terms. Note that the quantiles for the limit of the Brownian motion driven self-normalized Grenander-Rosenblatt statistic are explicitly known and therefore, have not to be estimated, see \eqref{distributionF}. For each setting, we derive the empirical $\alpha$-quantiles for $\alpha=0.9,0.95,0.975$ and $\alpha=0.99$. The results of the univariate CARMA(2,1) model are presented in \Cref{tab1_IP}, those of the bivariate Ornstein-Uhlenbeck model
in \Cref{tab2_IP} and those of the MCARMA(2,1) model in \Cref{tab3_IP}, respectively. The quantiles of the limit process in the NIG driven model differ from the associated ones in the Brownian motion driven model due to the additional term corresponding to $W_0^*$. However, that term is comparatively small so that the difference is unremarkable. 

%

\begin{table}[h!]
	\begin{center}
		\begin{tabular}{|c||c|c|c|c||c|c|c|c|}\hline
			\multicolumn{9}{|c|}{CARMA(2,1) process} \\\hline
			\multicolumn{9}{|c|}{self-normalized Grenander-Rosenblatt statistic} \\\hline
			&  \multicolumn{4}{|c||}{Gaussian distribution} & \multicolumn{4}{|c|}{NIG distribution}\\ \hline
			\hspace*{0.1cm} $n$ \hspace*{0.1cm} &90$\%$ & 95$\%$ & 97.5$\%$ & 99$\%$ &90$\%$ & 95$\%$ & 97.5$\%$ & 99$\%$    \\ \hline
			50&8.1129&9.4748&10.6459&12.2510&8.1277&9.5091&10.7964&12.6573 \\
			100&8.2701&9.6119&10.8839&11.9292&8.2646&9.6339&10.9753&12.5016 \\
			200&8.3499&9.6181&10.8841&12.4037&8.5780&9.8687&11.2801&12.4481\\
			500&8.5385&9.8123&10.9318&12.3058&8.5462&9.9177&11.1745&12.7717\\
			1000&8.6051&9.9091&11.1834&12.4889&8.5672&9.8623&11.0338&12.4965 \\
			
			limit&8.7067&9.9583&11.0970&12.4712&8.5691&9.8042&10.9474&12.2893\\  \hline \hline

			\multicolumn{9}{|c|}{self-normalized Cram\'er-von Mises statistic} \\\hline
			&  \multicolumn{4}{|c||}{Gaussian distribution} & \multicolumn{4}{|c|}{NIG distribution}\\ \hline
			\hspace*{0.1cm} $n$ \hspace*{0.1cm} &90$\%$ & 95$\%$ & 97.5$\%$ & 99$\%$ &90$\%$ & 95$\%$ & 97.5$\%$ & 99$\%$    \\ \hline
			50&72.2949&102.2965&130.2198&178.2448&72.5537&102.2986&137.1271&197.4643 \\
			100&71.0994&99.5237&132.2028&172.0599&74.4093&102.6841&132.2682&183.0207 \\
			200&73.8805&100.7941&130.0039&184.5601&75.3582&106.9690&138.9153&182.5730\\
			500&75.4037&105.8723&135.9596&176.5912&75.6696&107.4858&142.3017&178.4640\\
			1000&74.7563&104.4687&137.1544&172.4098&75.8411&105.2057&133.9230&176.4671 \\
			limit&73.6655&102.3420&131.5447&173.5202&75.2386&103.1234&132.1666&170.7825\\ \hline
			
		\end{tabular}
	\end{center}
	\begin{center}
		\caption{\label{tab1_IP}Empirical quantiles of the self-normalized Grenander-Rosenblatt and the self-normalized Cram\'er-von Mises statistic for the CARMA$(2,1)$ process. The estimated quantiles of the limit random variable are denoted as ``limit''. }
	\end{center}
	
\end{table}

\begin{table}[h!]
	\begin{center}
		\begin{tabular}{|c||c|c|c|c||c|c|c|c|}\hline
			\multicolumn{9}{|c|}{MCAR(1) process} \\\hline
			\multicolumn{9}{|c|}{self-normalized Grenander-Rosenblatt statistic} \\\hline
			&  \multicolumn{4}{|c||}{Gaussian distribution} & \multicolumn{4}{|c|}{NIG distribution}\\ \hline
			\hspace*{0.1cm} $n$ \hspace*{0.1cm} &90$\%$ & 95$\%$ & 97.5$\%$ & 99$\%$ &90$\%$ & 95$\%$ & 97.5$\%$ & 99$\%$    \\ \hline
			50&7.8557&9.1691&10.2976&11.5452&8.1736&9.4011&10.4213&11.8696 \\
			100&8.0828&9.2861&10.4627&11.9029&8.3682&9.6405&11.0312&12.3705 \\
			200&8.3262&9.6419&10.7303&12.2339&8.5911&9.9998&11.0681&12.6319\\
			500&8.5320&9.7029&10.8048&12.3653&8.6817&9.9822&11.3216&12.8713\\
			1000&8.6925&9.9823&10.9998&12.4769&8.7211&10.0636&11.3817&13.0281 \\
			
			limit&8.7067&9.9583&11.0970&12.4712&8.5754&9.8232&10.9469&12.3400\\  \hline \hline
			
			\multicolumn{9}{|c|}{self-normalized Cram\'er-von Mises statistic} \\\hline
			&  \multicolumn{4}{|c||}{Gaussian distribution} & \multicolumn{4}{|c|}{NIG distribution}\\ \hline
			\hspace*{0.1cm} $n$ \hspace*{0.1cm} &90$\%$ & 95$\%$ & 97.5$\%$ & 99$\%$ &90$\%$ & 95$\%$ & 97.5$\%$ & 99$\%$    \\ \hline
			50&69.5715&95.6165&124.5750&173.7440&71.9138&100.7472&130.5848&173.6131 \\
			100&71.7997&96.6962&125.7060&162.2857&75.1489&103.1489&135.8280&176.5584 \\
			200&72.6525&98.4600&130.0324&183.1381&78.5692&198.0443&141.3767&182.6919\\
			500&72.2218&102.6097&129.2060&171.7070&77.3160&105.7352&135.6170&180.1722\\
			1000&73.0592&101.7156&128.2745&163.3452&77.7714&106.5174&141.9770&185.5810 \\
			limit&73.8565&101.6551&130.6510&169.8042&74.3543&103.2249&132.9364&175.0584\\ \hline
			
		\end{tabular}
	\end{center}
	\begin{center}
		\caption{\label{tab2_IP}Empirical quantiles of the self-normalized Grenander-Rosenblatt and the self-normalized Cram\'er-von Mises statistic for the bivariate Ornstein-Uhlenbeck process. The estimated quantiles of the limit random variable are denoted as ``limit''. }
	\end{center}
	
\end{table}

\begin{table}[h!]
	\begin{center}
		\begin{tabular}{|c||c|c|c|c||c|c|c|c|}\hline
			\multicolumn{9}{|c|}{MCARMA(2,1) process} \\\hline
			\multicolumn{9}{|c|}{self-normalized Grenander-Rosenblatt statistic} \\\hline
			&  \multicolumn{4}{|c||}{Gaussian distribution} & \multicolumn{4}{|c|}{NIG distribution}\\ \hline
			\hspace*{0.1cm} $n$ \hspace*{0.1cm} &90$\%$ & 95$\%$ & 97.5$\%$ & 99$\%$ &90$\%$ & 95$\%$ & 97.5$\%$ & 99$\%$    \\ \hline
			50&8.2986&9.8566&11.2118&13.0063&8.5621&10.0598&11.5656&13.5447 \\
			100&8.4605&9.8910&11.0340&12.8121&8.4695&9.8199&10.9496&12.2852 \\
			200&8.6303&9.9647&11.2043&12.5652&8.5999&9.9500&11.1117&12.5176\\
			500&8.7450&9.9134&11.0104&12.4949&8.9067&10.2158&11.5731&13.2769\\
			1000&8.9541&10.3739&11.1723&12.5547&9.1116&10.4629&11.7347&13.1735 \\
			
			limit&8.7067&9.9583&11.0970&12.4712&8.5439&9.81128&10.8763&12.4449\\ \hline \hline
			
			\multicolumn{9}{|c|}{self-normalized Cram\'er-von Mises statistic} \\\hline
			&  \multicolumn{4}{|c||}{Gaussian distribution} & \multicolumn{4}{|c|}{NIG distribution}\\ \hline
			\hspace*{0.1cm} $n$ \hspace*{0.1cm} &90$\%$ & 95$\%$ & 97.5$\%$ & 99$\%$ &90$\%$ & 95$\%$ & 97.5$\%$ & 99$\%$    \\ \hline
			50&75.9319&108.3164&143.3946&209.7028&82.4033&115.1073&155.3495&225.8812 \\
			100&76.5426&105.9869&140.2673&188.0803&76.7595&105.1064&135.9159&182.0074 \\
			200&73.7611&102.1533&133.8212&177.4080&78.2793&110.1741&140.2286&184.8944\\
			500&79.5893&109.6364&144.6857&191.4346&80.8943&111.2767&143.9320&199.8840\\
			1000&76.5574&108.6585&138.5364&189.5779&79.9456&110.3208&144.0488&191.0527 \\
			limit&75.6944&103.1128&133.9440&173.5202&74.4492&102.3713&131.6506&171.6440\\ \hline
			
		\end{tabular}
	\end{center}
	\begin{center}
		\caption{\label{tab3_IP}Empirical quantiles of the self-normalized Grenander-Rosenblatt and the self-normalized Cram\'er-von Mises statistic for the MCARMA(2,1) process. The estimated quantiles of the limit random variable are denoted as ``limit''}
	\end{center}
	
\end{table}

\begin{table}[h!]
	\begin{center}
		\begin{tabular}{|c||c|c|c|c|c|c|c|}\hline
			\multicolumn{8}{|c|}{CARMA(2,1) process} \\\hline
			\multicolumn{8}{|c|}{self-normalized Grenander-Rosenblatt statistic} \\\hline
			&\multicolumn{7}{|c|}{Gaussian distribution} \\ \hline
			\hspace*{0.1cm} $n$ \hspace*{0.1cm} & T & (C1)& (C2) & (C3) & (C4) & (C5) & (C6)   \\ \hline
			50&3.74&58.06&99.98&100&99.98&77.02&83.30 \\
			100&4.06&93.14&100&100&100&95.44&97.94\\
			200&4.14&99.94&100&100&100&99.76&99.98\\
			500&4.50&100&100&100&100&100&100\\
			1000&4.80&100&100&100&100&100&100\\
			\hline \hline
			
			\multicolumn{8}{|c|}{self-normalized Cram\'er-von Mises statistic} \\\hline
			&  \multicolumn{7}{|c|}{Gaussian distribution} \\ \hline
			\hspace*{0.1cm} $n$ \hspace*{0.1cm} & T & (C1)& (C2) & (C3) & (C4) & (C5) & (C6)   \\ \hline
			50&4.92&61.00&99.84&100&99.94&71.00&74.16 \\
			100&4.54&93.68&100&100&100&92.82&94.18 \\
			200&4.72&99.94&100&100&100&99.72&99.84\\
			500&5.32&100&100&100&100&100&100\\
			1000 &5.08&100&100&100&100&100&100\\
			\hline

		\end{tabular}
	\end{center}
	\begin{center}
		\caption{\label{tab4_IP}Percentages of rejection of the self-normalized Grenander-Rosenblatt and the self-normalized Cram\'er-von Mises test statistics of the CARMA(2,1) model based on 5000 replications and the significance level $\alpha=0.05$. Thereby, ``T'' stands for the correct specified model whereas ``(C1)-(C6)'' denote
CARMA(2,1) models under the alternative. }
	\end{center}
	
\end{table}

\begin{table}[h!]
	\begin{center}
		\begin{tabular}{|c||c|c|c|c|c|}\hline
			\multicolumn{6}{|c|}{MCAR(1) process} \\\hline
			\multicolumn{6}{|c|}{self-normalized Grenander-Rosenblatt statistic} \\\hline
			&\multicolumn{5}{|c|}{Gaussian distribution} \\ \hline
			\hspace*{0.1cm} $n$ \hspace*{0.1cm} & T & (O1)& (O2) & (O3) & (O4)  \\ \hline
			50&3.06&60.74&94.00&61.50& 32.86\\
			100&3.64&86.98&99.72&92.12&55.24 \\
			200&4.30&98.96&100&100&83.54\\
			500&4.03&100&100&100&99.98\\
			1000 &4.80&100&100&100&100\\
			\hline \hline
			
			\multicolumn{6}{|c|}{self-normalized Cram\'er-von Mises statistic} \\\hline
			&\multicolumn{5}{|c|}{Gaussian distribution} \\ \hline
			\hspace*{0.1cm} $n$ \hspace*{0.1cm} & T & (O1)& (O2) & (O3) & (O4)  \\ \hline
			50&4.08&47.02&87.98&55.96&21.82 \\
			100&4.30&73.10&99.00&86.78&36.32 \\
			200&4.56&94.68&100&99.26&62.00\\
			500&4.92&100&100&100&94.42\\
			1000 &4.62&100&100&100&100\\
			\hline

		\end{tabular}
	\end{center}
	\begin{center}
		\caption{\label{tab5_IP}Percentages of rejection of the self-normalized Grenander-Rosenblatt and the
self-normalized Cram\'er-von Mises test statistics of the bivariate CAR(1) model based on 5000 replications and the significance level $\alpha=0.05$. Thereby, ``T'' stands for the correct specified model whereas ``(O1)-(O4)'' denote CAR(1) models under the alternative. }
	\end{center}
	
\end{table}

\begin{table}[h!]
	\begin{center}
		\begin{tabular}{|c||c|c|c|c|c|}\hline
			\multicolumn{6}{|c|}{MCARMA(2,1) process} \\\hline
			\multicolumn{6}{|c|}{self-normalized Grenander-Rosenblatt statistic} \\\hline
			&  \multicolumn{5}{|c||}{Gaussian distribution}\\ \hline
			\hspace*{0.1cm} $n$ \hspace*{0.1cm} & T & (M1)& (M2) & (M3) & (M4)  \\ \hline
			50&4.92&99.98&99.22&58.76&99.68 \\
			100&4.92&100&100&83.20&100 \\
			200&4.98&100&100&98.20&100\\
			500&4.80&100&100&100&100\\
			1000&5.66&100&100&100&100\\
			\hline \hline
			
			\multicolumn{6}{|c|}{self-normalized Cram\'er-von Mises statistic} \\\hline
			&  \multicolumn{5}{|c||}{Gaussian distribution}\\ \hline
			\hspace*{0.1cm} $n$ \hspace*{0.1cm} & T & (M1)& (M2) & (M3) & (M4) \\ \hline
			50&5.46&99.98&99.32&65.04&99.80 \\
			100&5.40&100&100&87.16&100 \\
			200&4.86&100&100&98.74&100\\
			500&5.72&100&100&100&100\\
			1000&5.66&100&100&100&100\\
			\hline

		\end{tabular}
	\end{center}
	\begin{center}
		\caption{\label{tab5_IP}Percentages of rejection of the self-normalized Grenander-Rosenblatt and the
self-normalized Cram\'er-von Mises test statistics of the MCARMA(2,1) models based on 5000 replications and the significance level $\alpha=0.05$. Thereby, ``T'' stands for the correct specified model whereas ``(M1)-(M4)'' denote MCARMA(2,1) models under the alternative. }
	\end{center}
	
\end{table}

As we can see, the quantiles of the test statistics are really similar to those of the limit processes even for small sample sizes in all settings. However, it should be mentioned that the results become worse when choosing bigger sample sizes. Precisely, when estimating quantiles for $n=2500$ and bigger, the empirical quantiles deviate from the limit quantiles slightly. Since a similar effect does not occur when choosing a remarkably smaller upper integration limit of the Grenander-Rosenblatt statistic
as, e.g.,
$ 
   (\sqrt{n}/\sqrt{m}) \sup_{t\in [0,1/2]}\left|\int_0^t  I_{n,Y^{(\Delta)}}(\omega)f^{(\Delta)}_Y(\omega^{-1}) d\omega-tm\right|
$
we lead the effect back to worse estimation results for frequencies close to $\pi$. 

In the following, we investigate the two test statistics under the hypothesis but as well under some alternatives. We test at the $5\%$ level. Therefore, as hypothesis we use the three models of above but the data is generated by different processes.  Namely, we consider the parametrization in the CARMA(2,1)
setting
\begin{align*}A&=\left( \begin{array}{c c} 0 & 1  \\ \vartheta_1&\vartheta_2 \end{array}\right),\quad \quad B=\left( \begin{array}{c}\vartheta_3\\ \vartheta_1+\vartheta_2\vartheta_3 \end{array}\right),
\qquad C=\left( \begin{array}{c c }1 & 0 \end{array}\right),\quad    \Sigma_{L}=1 \end{align*}
and choose the parameters
\begin{align*}
(C1)\ \vartheta&=(-1,-2,1), \qquad (C2)\ \vartheta=(-1,-2,-3),\\   (C3)\ \vartheta&=(-2,-3,5), \qquad
(C4)\ \vartheta=(-2,-1,-2),\\
(C5)\ \vartheta&=(-2,-1,-1), \ \quad (C6)\ \vartheta=(-1,-1,-1.5),
\end{align*}
for the generating processes.
In the same way, in the bivariate MCAR(1) setting, we consider the parametrization  \begin{align*}A&=\left( \begin{array}{c c} \vartheta_1 & \vartheta_2 \\ \vartheta_3&\vartheta_4 \end{array}\right)= B,
\qquad C=Id_2=\Sigma_{L},\end{align*}
with \begin{align*}
(O1)\ \vartheta&=(-1,-1/2,1,-1), \quad\qquad (O2)\ \vartheta=(-1/2,-1,1,-1),  \\
(O3)\ \vartheta&=(-1/2,-1/2,0,-1), \qquad (O4)\ \vartheta=(-1/2,-1/2,1,-2),
\end{align*} and finally, we take
\begin{align*}A&=\left( \begin{array}{c c c} \vartheta_1 & \vartheta_2&0  \\0&0&1\\\vartheta_3&\vartheta_4&\vartheta_5  \end{array}\right) , \qquad B=\left( \begin{array}{c c }\vartheta_1 & \vartheta_2  \\\vartheta_6&\vartheta_7\\\vartheta_3+\vartheta_5\vartheta_6&\vartheta_4+\vartheta_5\vartheta_7  \end{array}\right), \quad
C=\left( \begin{array}{c c c}1 & 0& 0 \\0&1&0  \end{array}\right), \quad \Sigma_L=Id_2,  \end{align*}
in the MCARMA(2,1) setting with the parameters \begin{align*}
(M1)\ \vartheta&=(-2,1,-3,-1,-1,1,1), \quad\qquad (M2)\ \vartheta=(-2,-1,3,-1,-3,-1,-3),  \\
(M3)\ \vartheta&=(-1,5,-1,0,-3,-1,-1), \  \qquad (M4)\ \vartheta=(-1,4,-2,0,-3,-1,-1),
\end{align*}
to generate data under different alternatives.
The results are presented in \Cref{tab4_IP} - \Cref{tab5_IP}. As suspected, in the correct specified setting, the statistics hold the given level for most sample sizes. Under the alternatives, the statistics reject quite often for moderate sample sizes and detect every alternative with certainty for $n=1000$ and higher. The performances of the self-normalized Grenander-Rosenblatt and Cram\'er-von Mises statistic seem to be comparable.

\section{Proofs}    \label{sec:Proofs}

Let $E_{n,N^{(\Delta)}}$ be the empirical spectral process based on the i.i.d. sequence
$(N_k^{(\Delta)})_{k\in\Z}$, see \eqref{emp_noise} below for an explicit definition.
The aim is to decompose $\tr(E_{n}(g))$ as in \eqref{eq15} in
$$\tr(E_{n}(g))=\tr(E_{n,N^{(\Delta)}}(g^\Phi))+E_{n,R}(g)$$
 and show that $E_{n,R}$ is asymptotically neglectable. Then, the asymptotic behavior of
$\tr(E_{n}(g))$ is determined by the asymptotic behavior of $\tr(E_{n,N^{(\Delta)}}(g^\Phi))$.
Thus,  in \Cref{Section_5.1}, we first show the asymptotic behavior of the empirical spectral process $E_{n,N^{(\Delta)}}$  as a kind of special case of  \Cref{Satz::intP::Y}.
Then, in \Cref{Section:5.2}, we derive that the error $E_{n,R}$ which occurs by approximating the original process by the white noise process
has no influence on the limit behavior.

\subsection{The functional central limit theorem for the white noise process} \label{Section_5.1}

\subsubsection{Preliminaries}
For $s\geq 0$ we define the space
 $\mathcal H_{*,r}^s:=\{g\in \mathcal H_r|\,\,\, \|g\|_{*,s}<\infty\}$ with
\beao
    \|g\|_{*,s}^2:=\|g\|_{Id_r,s}^2:=\sum_{h=-\infty}^\infty(1+|h|)^{2s}\|\widehat g_h\|^2,
\eeao
which is a normed space.
Again, $\|g\|_{*,0}=\|g\|_{2}$.
In the case that $\mathcal{G}_r$ is additionally a vector space, we consider the dual space
\beao
    {\mathcal{G}_r^*}':=\{F: \mathcal G_r\to \C^{r\times r}| \,\,\,F \text{ is linear }\}
\eeao
with the operator norm
\beao
    \|F\|_{{\mathcal{G}_r^*}'}^s:=\sup_{g\in \mathcal{G}_r \atop {\|g\|_{*,s}\leq 1}}\|F(g)\| \quad \text{ for } F\in {\mathcal{G}_r^*}',
\eeao
which generates the metric $d_{{\mathcal{G}_r^*}'}^s$.

 \subsubsection{The functional central limit theorem for the white noise process: main results}
 We  introduce the assumptions to derive the asymptotic behavior of the empirical spectral process for
 the white noise process which correspond to those of Assumption B .
 \begin{assumptionletter2}\label{Assumption N} $\mbox{}$\\
 	Let $\mathcal G_r\subseteq \mathcal H_r$, $h\in\mathcal H_r$ and
    suppose $\|g\|_{*,s}<\infty$ for any $g\in \mathcal G_r$ and some $s\geq 0$.	Suppose that
 $\mathcal{G}_r$ satisfies either (B1), (B2) or (B3).
 \end{assumptionletter2}
Assumption N is not requiring $\mathcal G_r$ to be totally bounded, but in $(B2)$ and $(B3)$ this is already satisfied.

 \begin{theorem} \label{Theorem:5.2}
  Let Assumption $A$ and $N$ hold. Furthermore, suppose that $\sup_{g\in\mathcal{G}_r}\|g\|_{*,s}<\infty$
  and  $\mathcal G_r$ is totally bounded.
  Define
  \beam \label{emp_noise}
      E_{n,N^{(\Delta)}}(g):=\sqrt n\int_{-\pi}^{\pi}g(\omega)\left(I_{n,N^{(\Delta)}}(\omega)-\frac 1 {2\pi }\Sigma_N^{(\Delta)}\right)d\omega \quad \text{ for }g\in \mathcal G_r.
  \eeam
   Then, as $n\to\infty$,
 	\begin{align*}
    (E_{n,N^{(\Delta)}}(g))_{g\in\mathcal G_r}
 	      \overset{\mathcal D}{\longrightarrow}
    (E_N(g))_{g\in\mathcal G_r}\quad \text{in } (\mathcal C (\mathcal G_{r},\C^{r\times r}),d_{\mathcal G_{r}}),
 \end{align*}
 where
 \begin{align*}
 E_N(g)&=\widehat g_0W_0'+\sum_{h=1}^{\infty}\left(\widehat g_hW_h+\widehat g_{-h} W_{h}^\top\right)
 \end{align*}
 and $W_0'$ and $(W_h)_{h\in\Z}$ are defined as in \Cref{Satz::intP::Y2}.
 \end{theorem}
Of course, $\sup_{g\in\mathcal{G}_r}\|g\|_{*,s}<\infty$ implies $\sup_{g\in\mathcal{G}_r}\|g\|_{2}<\infty$.
Similarly, we obtain the following result in the dual space.

 \begin{theorem}\label{Satz::intP::N}~\\
 	Let Assumption $A$ and $N$ hold. Furthermore, suppose that $\mathcal G_r$ is a Hilbert space. 
 Then, as $n\to\infty$,
 	\begin{align*}
    (E_{n,N^{(\Delta)}}(g))_{g\in\mathcal G_r}
 	      \overset{\mathcal D}{\longrightarrow}
    (E_N(g))_{g\in\mathcal G_r}\quad \text{  in }  ({\mathcal G_r^*}',d_{{\mathcal G_r^\ast}'}^s).
 \end{align*}
 \end{theorem}
Note,  the assumption of totally boundedness of $\mathcal G_r$ is not necessary. Furthermore,
$\mathcal{G}_r$ as defined in (B3) is not a Hilbert space and hence, not covered in this theorem.

 It is well known that a sequence of probability measures in some Banach space converges weakly if it is tight in the weak topology and if the finite dimensional distributions converge.
 Therefore, we first prove the weak convergence of the finite dimensional distributions of $E_{n,N^{(\Delta)}}$.

 \subsubsection{The functional central limit theorem for the white noise process: Convergence of the finite dimensional distributions }
 \begin{lemma}\label{fidiCon}~\\
 	Let Assumption $N$ hold and $\E\|L_1\|^4<\infty$. Then, for $ k\in \N$, $g_1,\ldots,g_k \in \mathcal H_r$ we have
 $$(E_{n,N^{(\Delta)}}(g_1),\ldots, E_{n,N^{(\Delta)}}(g_k))\overset{\mathcal D}{\longrightarrow}(E_N(g_1),\ldots,E_N(g_k)) \quad \text{ in } \C^{r\times rk}.$$
 \end{lemma}
 \begin{proof}
 	Let $C_1,\ldots,C_k\in\C^{r\times r}$. Then,
    \beao
        \sum_{j=1}^k\text{vec}(C_j^\top)^\top\text{vec}(E_{n,N^{(\Delta)}}(g_j))=\tr\left(E_{n,N^{(\Delta)}}\left(\sum_{j=1}^kC_jg_j\right)\right)
    \eeao
    and $\sum_{j=1}^kC_jg_j\in\mathcal H_r$.
    Thus,  it is sufficient to prove that
     $$E_{n,N^{(\Delta)}}(g)\overset{\mathcal D}{\longrightarrow}E_{N}(g) \quad \text{ in } \C^{r\times r} \quad \text{ for any } g\in \mathcal H_r. $$
 	Further, the representation
 	\begin{eqnarray}
 	\nonumber E_{n,N^{(\Delta)}}(g)
 	\nonumber&=&\sqrt n\int_{-\pi}^{\pi}\sum_{\ell=-\infty}^\infty\widehat g_\ell e^{i\omega \ell}\left(\frac{1}{2\pi n}\sum_{j,k=1}^{n}N_j^{(\Delta)}N_k^{(\Delta)\top}e^{-i\omega(j-k)}-\frac 1 {2\pi }\Sigma_N^{(\Delta)}\right)d\omega\\
 	\nonumber&=&\frac{1}{ \sqrt n}\sum_{k=1}^{n}\sum_{\ell=1-k}^{n-k}\widehat g_{\ell}N_{\ell+k}^{(\Delta)}N_k^{(\Delta)\top}-\sqrt n\widehat g_0 \Sigma_N^{(\Delta)}\\
 %
 	  &=&\sqrt n\widehat g_0(\overline \Gamma_{n,N^{(\Delta)}}(0)-\Sigma_{N^{(\Delta)}})+\sqrt n\left( \sum_{h=1}^{n-1}\widehat g_h\overline \Gamma_{n,N^{(\Delta)}}(h)+\widehat g_{-h} \overline \Gamma_{n,N^{(\Delta)}}(h)^\top\right) \label{aux_rep}
 	\end{eqnarray}
 	holds. We fix an upper bound for $h$, say $M$, and apply Lemma 6 in \cite{Fasen:Mayer:2020:Supp}.
 	Thus, we have as $n\to\infty$,
 \begin{eqnarray*}
 	\lefteqn{\hspace*{-2cm}\sqrt n\widehat g_0(\overline \Gamma_{n,N^{(\Delta)}}(0)-\Sigma_{N^{(\Delta)}})+\sqrt n \sum_{h=1}^{M}\left(\widehat g_h\overline \Gamma_{n,N^{(\Delta)}}(h)+\widehat g_{-h} \overline \Gamma_{n,N^{(\Delta)}}(h)^\top\right)}\\
    &&\hspace*{-2cm}\overset{\mathcal D}{\longrightarrow} \widehat g_0W_0'+\sum_{h=1}^{M}\left(\widehat g_hW_h+\widehat g_{-h} W_h^\top\right).
 	\end{eqnarray*}
  	In view of Proposition 6.3.9  of \cite{Brockwell:Davis:1991}, it remains to prove
  \begin{align*}
 \lim_{M\to\infty}\limsup_{n\to \infty} \P\left(\left\|\sqrt n \sum_{h=M+1}^{n-1}\widehat g_h\overline \Gamma_{n,N^{(\Delta)}}(h)+\widehat g_{-h} \overline \Gamma_{n,N^{(\Delta)}}(h)^\top\right\|>\varepsilon\right)=0.\end{align*}
 	Tschebycheffs inequality leads to \begin{eqnarray*}\lefteqn{ \P\left(\left\|\sqrt n \sum_{h=M+1}^{n-1}\widehat g_h\overline \Gamma_{n,N^{(\Delta)}}(h)+\widehat g_{-h} \overline \Gamma_{n,N^{(\Delta)}}(h)^\top\right\|>\varepsilon\right)}\\
 		&&\leq \frac{n}{\varepsilon^2}\E\left[\left\| \sum_{h=M+1}^{n-1}\widehat g_h\overline \Gamma_{n,N^{(\Delta)}}(h)+\widehat g_{-h} \overline \Gamma_{n,N^{(\Delta)}}(h)^\top\right\|^2\right]. \end{eqnarray*}
 	Since $(N_k^{(\Delta)})_{k\in \Z}$ is i.i.d., $\overline\Gamma_{n,N^{(\Delta)}}(h)$ and $\overline\Gamma_{n,N^{(\Delta)}}(j)$ are uncorrelated for $h\neq j$.
 Due to the use of the Frobenius norm we  get
 \begin{eqnarray*}
 	\E\left[\left\| \sum_{h=M+1}^{n-1}\left(\widehat g_h\overline \Gamma_{n,N^{(\Delta)}}(h)+\widehat g_{-h} \overline \Gamma_{n,N^{(\Delta)}}(h)^\top\right)\right\|^2\right]
 		\leq 2\sum_{h=M+1}^{n-1}\left(\|\widehat g_h\|^2 +\|\widehat g_{-h}^\top\|^2\right) \E\left[\|\overline \Gamma_{n,N^{(\Delta)}}(h)\|^2\right].
 	\end{eqnarray*}
 	On the one hand,
    \begin{align*}
 	  \E\left[n \| \overline \Gamma_{n,N^{(\Delta)}}(h)\|^2\right]=\frac 1 n \sum_{S,T=1}^{r}\sum_{k,\ell=1}^{n-h}\E\left[\left(N_{k+h}^{(\Delta)}N_{k}^{(\Delta)\top}N_{\ell+h}^{(\Delta)}N_{\ell}^{(\Delta)\top}\right)[S,T]\right]
    \leq\frac{n-h}{n}\mathfrak{C},
 	\end{align*}
 	where $\mathfrak{C}$ is a constant which is independent of $h$.
 	On the other hand, Parseval's equality yields then 
    \beao
         \lim_{M\to\infty}\limsup_{n\to \infty} \sum_{h=M+1}^{n-1}\left(\|\widehat g_h\|^2 +\|\widehat g_{-h}^\top\|^2\right) \E\left[\|\overline \Gamma_{n,N^{(\Delta)}}(h)\|^2\right]=0,
    \eeao
    and the proof is completed.
 \end{proof}
 \subsubsection{The functional central limit theorem for the white noise process: Tightness}
 To prove \Cref{Satz::intP::N} and \Cref{Theorem:5.2}, respectively, it remains to show the tightness of $(E_{n,N^{(\Delta)}})_{n\in \N}$.
 Therefore, it is important that $(\mathcal{C}(\mathcal G_r,\C^{r\times r}),d_{\mathcal G_r})$ (if $\mathcal G_r$ is
 totally bounded) and $({\mathcal{G}_r^*}',d_{{\mathcal G_{r}^*}'}^{s})$ are Banach spaces.

 In \Cref{Satz::intP::N}, we additionally assumed that $\mathcal G_r$ is a Hilbert space. Thus, ${\mathcal{G}_r^*}'$
 is a Hilbert space as well and
   the set $\{\Phi_g: {\mathcal{G}_r^*}' \to\C^{r\times r} | g\in \mathcal G_r, \Phi_g(F)=F(g)\quad \forall \,F\in {\mathcal{G}_r^*}' \}$ is the  dual space of
    ${\mathcal{G}_r^*}'$. 
  Due to Theorem~2.3 in \cite{DeAcosta} (cf. \cite{Ledoux:Talagrand}) 
 and the convergence of the finite-dimensional distributions, it is then sufficient to show that $({E_{n,N^{(\Delta)}}})_{n\in \N}$
 is flatly concentrated, i.e., for any $\varepsilon$, $\delta>0$ there exists a finite-dimensional subspace $L\subset  {\mathcal{G}_r^*}'$
 with
 $$\inf_{n\in \N}\P(d_{{\mathcal G_{r}^*}'}^{s}(E_{n,r},L)\leq \delta )=\inf_{n\in \N}
        \P(\left\|\text{Pr}_{L^\perp}(E_{n,N^{(\Delta)}})\right\|_{{\mathcal{G}_r^*}'}^s\leq \delta ) \geq 1-\varepsilon.$$
 Therefore, we choose a sequence $L_M\subseteq {\mathcal{G}_r^*}'$, $M\in\N$, with
 \begin{align*}
    \lim_{M\to\infty}\limsup_{n\to \infty}\P\left(\left\|\text{Pr}_{L_M^\perp}(E_{n,N^{(\Delta)}})\right\|_{{\mathcal{G}_r^*}'}^s> \delta \right)=0.
 \end{align*}
 	For $M\in \N$, we define $L_M\subseteq {\mathcal{G}_r^*}'$ as the linear subspace generated by
 $(e_h)_{|h|\leq M}$ where $e_h: \mathcal{G}_r\to\C^{r\times r}$ is defined as $e_h(g)=\widehat g_h .$
 	Then, due to \eqref{aux_rep}, we obtain
    \begin{eqnarray*}
        \left\|\text{Pr}_{L_M^\perp}(E_{n,N^{(\Delta)}})\right\|_{{\mathcal{G}_r^*}'}^s\leq \sup_{\substack{g\in \mathcal G_r\\
\|g\|_{*,s}\leq 1  \\\widehat g_h=0,
        |h|\leq M}}\|E_{n,N^{(\Delta)}}(g)\|^2
 		&\leq & \sup_{\substack{g\in \mathcal G_r} \atop
\|g\|_{*,s}\leq 1}  \left\| \sqrt{n} \sum_{M<|h|\leq n}\widehat g_h\left(\overline \Gamma_{n,N^{(\Delta)}}(h)-\E[\overline \Gamma_{n,N^{(\Delta)}}(h)]\right)\right\|^2.
 \end{eqnarray*}
 Finally, since $\E[\overline \Gamma_{n,N^{(\Delta)}}(h)]=0$ for $h\not=0$, the sequence $(E_{n,N^{(\Delta)}})_{n\in\N}$ is flatly concentrated
 in  $({\mathcal{G}_r^*}',d_{{\mathcal G_{r}^*}'}^{s})$ if
\begin{align}\label{unibound1}
 \lim_{M\to\infty}\limsup_{n\to \infty} \P\left(\sup_{g\in \mathcal G_r \atop
\|g\|_{*,s}\leq 1}\sqrt{n}\left\| \sum_{M<|h|\leq n}\widehat g_h \overline \Gamma_{n,N^{(\Delta)}}(h)\right\|>\delta\right)=0 \qquad \forall\ \delta>0.
 \end{align}
 As consequence, for the proof of \Cref{Satz::intP::N} it remains to show \eqref{unibound1}. How is it in \Cref{Theorem:5.2}?

 In \Cref{Theorem:5.2}, we additionally assumed that $\mathcal G_r$ is totally bounded such that
 $(\mathcal{C}(\mathcal G_r,\C^{r\times r}),d_{\mathcal G_r})$  is a separable  Banach space.  Similarly as above
  a sufficient condition for $(E_{n,N^{(\Delta)}})_{n\in \N}$  to be flatly concentrated in $(\mathcal{C}(\mathcal G_r,\C^{r\times r}),d_{\mathcal G_r})$ is
 \begin{align}\label{unibound}
 \lim_{M\to\infty}\limsup_{n\to \infty} \P\left(\sup_{g\in \mathcal G_r}\sqrt{n}\left\| \sum_{M<|h|\leq n}\widehat g_h \overline \Gamma_{n,N^{(\Delta)}}(h)\right\|>\delta\right)=0 \qquad \forall\ \delta>0.
 \end{align}
 Due to \cite{BharuchaReidRoemisch}, Proposition 2.2, on a separable Banach space tightness is equivalent to flat concentration and
 uniform boundedness. But under the assumptions of \Cref{Theorem:5.2}, in $(\mathcal{C}(\mathcal G_r,\C^{r\times r}),d_{\mathcal G_r})$ the sequence $(E_{n,N^{(\Delta)}})_{n\in\N}$ is uniformly bounded if it is flatly concentrated:

  \begin{lemma} \label{aux2}
 Let Assumption A and  $\sup_{g\in \mathcal G_{r}}\|g\|_{2}<\infty$ hold.
 If \Cref{unibound} is satisfied then $(E_{n,N^{(\Delta)}})_{n\in\N}$
    is uniformly bounded in  $(\mathcal{C}(\mathcal G_r,\C^{r\times r}),d_{\mathcal G_r})$.
 \end{lemma}
For completeness, the proof is provided in the Appendix.

  In summary, for the proof of tightness in \Cref{Theorem:5.2} and \Cref{Satz::intP::N}, respectively it remains to show that $(E_{n,N^{(\Delta)}})_{n\in\N}$
  is flatly concentrated  which is given by  \eqref{unibound} and \eqref{unibound1}, respectively.


  \begin{lemma}\label{tightness::N1}	~\\
 	Let Assumption $A$ and Assumption $(B1)$ hold.
    \begin{itemize}
        \item[(a)] Suppose $\sup_{g\in\mathcal{G}_r}\|g\|_{*,s}<\infty$ and $\mathcal G_r$ is totally bounded.  Then, the sequence $(E_{n,N^{(\Delta)}})_{n\in \N}$
        satisfies \eqref{unibound}, and is flatly concentrated and tight in $(\mathcal{C}(\mathcal G_{r},\C^{r\times r}),
 d_{\mathcal G_{r}}).$
                \item[(b)] Suppose $\mathcal{G}_r$ is a Hilbert space.  Then, the sequence $(E_{n,N^{(\Delta)}})_{n\in \N}$ satisfies \eqref{unibound1} and is
        flatly concentrated and tight in $({\mathcal{G}_r^*}',d_{{\mathcal G_{r}^*}'}^{s})$.
        \end{itemize}
 \end{lemma}
 \begin{proof} $\mbox{}$
 	\begin{itemize}
 		\item[(a)]
  Due to Markov's inequality  we receive
    \beao
        \lefteqn{\P\left(\sup_{g\in \mathcal G_r}\left\|\sqrt{n} \sum_{M<|h|\leq n}\widehat g_h \overline \Gamma_{n,N^{(\Delta)}}(h)\right\|>\delta\right)}\\
            &&\leq \frac{1}{\delta^2}\E\left(\sup_{g\in \mathcal G_r}n\left\| \sum_{M<|h|\leq n}\widehat g_h \overline \Gamma_{n,N^{(\Delta)}}(h)\right\|^2\right)\\
                      &&\leq \frac{1}{\delta^2}\sup_{g\in\mathcal{G}_r}\sum_{h=-\infty}^\infty (1+|h|)^{2s}\|\widehat g_h\|^2n\E\left( \sum_{M<|h|\leq n}(1+|h|)^{-2s}\left\|\overline \Gamma_{n,N^{(\Delta)}}(h)\right\|^2\right)\\
            &&\leq \frac{1}{\delta^2}\sup_{g\in\mathcal{G}_r}\|g\|_{*,s}^2 \sum_{M<|h|\leq n}(1+|h|)^{-2s} \sup_{M<|h|\leq n}\E\left(n\|\overline\Gamma_{n,N^{(\Delta)}}(h)\|^2\right).
    \eeao
 	Due to \Cref{aux3} 
  and $\sum_{M<|h|} (1+|h|)^{-2s}\overset{M\to \infty}{\longrightarrow}0$ for $s>1/2$, the convergence \eqref{unibound}
    follows.
 \item[(b)] The proof is analogue to the proof of (a) by replacing $\sup_{g\in \mathcal G_r }$ by $\sup_{g\in \mathcal G_r \atop
\|g\|_{*,s}\leq 1}$.
	\end{itemize}
 \end{proof}

 \begin{lemma}\label{tightness::N3}~\\
 	Let Assumption $A$ and  Assumption $(B2)$ hold.
     \begin{itemize}
        \item[(a)] Suppose $\sup_{g\in\mathcal{G}_r}\|g\|_{*,0}<\infty$.  Then, the sequence $(E_{n,N^{(\Delta)}})_{n\in \N}$ satisfies \eqref{unibound}, and is
        flatly concentrated and tight in $(\mathcal{C}(\mathcal G_{r},\C^{r\times r}),
 d_{\mathcal G_{r}}).$
                \item[(b)] Suppose $\mathcal{G}_r$ is a Hilbert space.  Then, the sequence $(E_{n,N^{(\Delta)}})_{n\in \N}$ satisfies \eqref{unibound1}, and is flatly concentrated and tight in $({\mathcal{G}_r^*}',d_{{\mathcal G_{r}^*}'}^{s})$.
    \end{itemize}
 \end{lemma}

 For the proof we require some auxiliary lemma.

 \begin{lemma} \label{Lemma:A}
 Suppose there exists a constant $K_L>0$ such that the joint cumulant satisfies \linebreak $|\text{cum}(BL_1[k_1],\ldots,BL_1[k_j])| \leq K_L^j$ for $k_1,\ldots,k_j\in\{1,\ldots,r\}$
 and  $j\in \N$.
 Define  for $M\in \N_0$
\beao
    E_{n,N^{(\Delta)}}^{(M)}(g):=\sqrt n \sum_{M<|h|\leq n}\widehat g_h \overline \Gamma_{n,N^{(\Delta)}}(h) \quad  \text{ for } g\in \mathcal H_r.
\eeao
  Then,
 there exist some constants $c_1,c_2>0$ such that for any  $\delta>0$ and $M\in\N_0$ we have
 \beam \label{A2}
      \P\left(\|E_{n,N^{(\Delta)}}^{(M)}(g)\|>\delta\right)\leq c_1\exp\left(-c_2\sqrt{\frac{\delta}{\sqrt{\sum_{M<|h|\leq n}\|\widehat g_h\|^2}}}\right) \quad \text{ for } g\in \mathcal H_r,\,n\in\N.
 \eeam
 \end{lemma}
 \begin{proof}
 	We prove that Assumption (2.1) (a) of \cite{dahlhaus1988empirical} is satisfied.
 Therefore, define  the $k$-th order cumulant spectrum $f_{k_1,\ldots, k_j}$ of $(N^{(\Delta)}_{k}[k_1],\ldots,N^{(\Delta)}_{k}[k_{j}])_{k\in\N}$ for  $k_1,\ldots, k_j \in \{1,\ldots, r\}$ and $j\in \N$
    as  in \cite{Brillinger}, p.~25, and show
  that there exists some constant $K_f>0$ such that $$|f_{k_1,\ldots, k_j}(\lambda_1,\ldots, \lambda_{j-1})|\leq K_f^j$$ for all  $ k_1,\ldots, k_j \in \{1,\ldots, r\}$,
  $ \lambda_1,\ldots, \lambda_{j-1}\in \R$ and $j\in \N$.  Since $(N_k^{(\Delta)})_{k\in\N}$ is an i.i.d. sequence, we have
  \begin{eqnarray*}
 f_{k_1,\ldots, k_j}(\lambda_1,\ldots,\lambda_{j-1})
 	=\frac{1}{(2\pi)^{j-1}}\text{cum}(N^{(\Delta)}_{1}[k_1],\ldots,N^{(\Delta)}_{1}[k_{j-1}],N^{(\Delta)}_1[k_j]).\end{eqnarray*}
   It remains to show that there exists some constant $K_N>0$ such that $$|\text{cum}(N^{(\Delta)}_{1}[k_1],\ldots,N^{(\Delta)}_1[k_j])|\leq K_N^j
   \quad \text{ for }  k_1,\ldots, k_j\in\{1,\ldots,r\} \text{ and } j\in \N.$$
   Note,  $|\text{cum}(BL_1[k_1],\ldots,BL_1[k_j])| \leq K_L^j$ means that
    the  cumulant generating function $C_{BL_1[k_1],\ldots,BL_1[k_j]}$ of $BL_1[k_1],\ldots,BL_1[k_j]$ satisfies
    $$\left|\left.\frac{\partial^j}{\partial u_{1}\ldots \partial u_{j}}C_{BL_1[k_1],\ldots,BL_1[k_j]}(u_1,\ldots,u_j)\right|_{(u_1,\ldots,u_j)=(0,\ldots,0)}\right|\leq K_L^j.$$
   The cumulant generating function $C_{N_1[k_1],\ldots, N_1[k_j]}$ of  $N_1^{(\Delta)}[k_1],\ldots, N_1^{(\Delta)}[k_j]$ is
 	\begin{eqnarray*}
    C_{N_1[k_1],\ldots, N_1[k_j]}(u_1,\ldots,u_j)&=&\log\left(\E\left[\exp\left(i(u_1N_1^{(\Delta)}[k_1]+\ldots+u_jN_1^{(\Delta)}[k_j])\right)\right]\right)\\
    &=&\int_0^\Delta C_{(BL_1)^\top,\ldots,(BL_1)^\top}\left(u_1(e_{k_1}^\top e^{As})^\top,\ldots,u_j(e_{k_j}^\top e^{As})^\top\right)ds,
 	\end{eqnarray*}
    see \cite{Rajput:Rosinski}, where $e_k\in\R^{r}$ denotes the unit vector which is $1$
    in the $k$-th component  and $0$ otherwise.
 	Accordingly, we have 
    \begin{eqnarray*} \label{eqref13}
 	\lefteqn{\text{cum}(N_1^{(\Delta)}[k_1],\ldots,N_1^{(\Delta)}[k_j])} \nonumber\\
    &&=\left.\frac{\partial^j}{\partial u_{1}\ldots \partial u_{j}}
    \int_0^\Delta C_{(BL_1)^\top,\ldots,(BL_1)^\top}\left(u_1(e_{k_1}^\top e^{As})^\top,\ldots,u_j(e_{k_j}^\top e^{As})^\top\right)ds\right|_{(u_1,\ldots,u_j)=(0,\ldots,0)}.
 	\end{eqnarray*}
   Interchanging differentiation and integration due to dominated convergence 
   yields
 	\begin{align*}
 	|\text{cum}(N_1^{(\Delta)}[k_1],\ldots,N_1^{(\Delta)}[k_j])|\leq \left(K_0
 	\max_{s\in [0,\Delta]}\|e^{As}\|\right)^j(K_L)^{rj},
 	\end{align*}
 	where $K_0$ is a constant which is independent of $j$ and $k_1,\ldots k_j$.
 Therefore, Assumption (2.1) (a) of \cite{dahlhaus1988empirical} is satisfied and the proof of \Cref{Lemma:A} matches the proof of Lemma 2.3 in \cite{dahlhaus1988empirical}.
 \end{proof}

\begin{lemma} \label{Lemma 5.9}
Let $\mathcal G_r\subseteq \mathcal H_r$ be totally bounded and $\widetilde{\mathcal G_r}\subseteq \mathcal G_r$
with $\sup_{g\in  \widetilde{\mathcal G_r}}\|g\|_2<\infty$. Suppose there exists a constant
$K_L>0$ such that the joint cumulant $\text{cum}(BL_1[k_1],\ldots,BL_1[k_j]) \leq K_L^j$ for all  $k_1,\ldots,k_j\in\{1,\ldots,r\}$ and $j\in \N$, and
  $$\int_0^1 [\log(N(\varepsilon,\mathcal G_r,d_2))]^2d\varepsilon<\infty.$$
Let $E_{n,N^{(\Delta)}}^{(M)}$ be defined as in \Cref{Lemma:A}.
Then, there exists a set $B_n$ (independent of $\widetilde {\mathcal G _r}$) with \linebreak $\lim_{n\to\infty}\p(B_n)=1$ and some constants
$c_1,c_2>0$ such that for any $\delta>0$ there exists a $M_0\in\N$
with
\beao \label{A4} \label{A3}
     \P\left(\sup_{g \in\widetilde{\mathcal G_r}}\|E_{n,N^{(\Delta)}}^{(M)}(g)\|>\delta,B_n\right)\leq c_1\exp\left(-c_2\sqrt{\frac{\delta}{\sup_{g\in  \widetilde {\mathcal G_r}}\sqrt{\sum_{M<|h|\leq n}\|\widehat g_h\|^2}}}\right) \quad \forall\, M\geq M_0,\, n\in\N.
\eeao
\end{lemma}
\begin{proof}
The proof  goes in the same way as the proof of inequality (28)  in \cite{dahlhaus2009empirical} using \eqref{A2}
and $\lim_{M\to\infty}\sup_{g \in\widetilde{\mathcal G_r}}\sum_{|h|>M}\|\widehat g_h\|^2=0$.
\end{proof}

 \noindent\textit{Proof of \Cref{tightness::N3}.}
The proof of \eqref{unibound} and  \eqref{unibound1}, respectively follow directly from \Cref{Lemma 5.9} since $\lim_{n\to\infty}\P(B_n)=1$.
\hfill$\Box$


 \begin{lemma}\label{tightness::N2}~\\
 	Let Assumption $A$ and  Assumption $(B3)$ hold.
       The sequence $(E_{n,N^{(\Delta)}})_{n\in \N}$ is tight in $(\mathcal{C}(\mathcal G_{r},\C^{r\times r}),
 d_{\mathcal G_{r}}).$
 \end{lemma}
 \begin{proof} $\mbox{}$
 Following the lines of the proof of Theorem 3.2 in \cite{Kluppelberg:1996}, but without using the contraction principle on p.~1875, and
 using the Markov inequality  on p.~1876 with $\mu=2$ instead of Theorem 6.9.4 of \cite{Kwapien},  we obtain
 \beao
     \lim_{M\to\infty}\limsup_{n\to\infty}\P\left(\sup_{t\in[-\pi,\pi]}\|E_{n,N^{(\Delta)}}^{(M)}(Id_r\1_{\left[-\pi,t\right]}(\cdot))\|>\delta\right)=0
 \eeao
 (see as well the arguments in \cite{Kokoszka:Mikosch}, proof of Theorem~5.1). That results in the flat concentration
 condition on  $(\mathcal{C}(\mathcal G_{r}^F,\C^{r\times r})$
 with ${\mathcal{G}}_r^F:=\{Id_r\1_{\left[-\pi,t\right]}(\cdot):t\in[-\pi,\pi]\}$.
 Due to the convergence of the finite-dimensional distribution in \Cref{fidiCon} and the flat concentration
 on  $(\mathcal{C}(\mathcal G_{r}^F,\C^{r\times r})$ we receive
 the weak convergence as $n\to\infty$,
 \begin{align*}
    (E_{n,N^{(\Delta)}}(g))_{g\in{\mathcal G}_r^F}
 	      \overset{\mathcal D}{\longrightarrow}
    (E_{N}(g))_{g\in{\mathcal G}_r^F}\quad \text{ in } (\mathcal{C}({\mathcal G}_r^F,\C^{r\times r}),d_{ {\mathcal G}_r^F}).
 \end{align*}
 Since $h\in\mathcal{H}_r$ is continuously differentiable on the interior of its support, partial integration and the continuous mapping theorem result in
 the weak convergence as $n\to\infty$,
 \begin{align*}
    (E_{n,N^{(\Delta)}}(g))_{g\in{\mathcal G}_r}
 	      \overset{\mathcal D}{\longrightarrow}
    (E_N(g))_{g\in{\mathcal G}_r}\quad \text{ in } (\mathcal{C}({\mathcal G}_r,\C^{r\times r}),d_{{\mathcal G}_r}),
 \end{align*}
 and in particular, the tightness.
 \end{proof}

\subsection{Proof of \Cref{Satz::intP::Y2}} \label{Section:5.2}

 To deduce \Cref{Satz::intP::Y} from \Cref{Satz::intP::N}, we have to check that the error
  $$E_{n,R}(g)=\sqrt n \int_{-\pi}^{\pi}\tr\left(g(\omega)\left(I_{n,Y^{(\Delta)}}(\omega)-\Phi(e^{-i\omega})I_{n,N^{(\Delta)}}(\omega)\Phi(e^{i\omega})^\top \right)\right)  d\omega ,$$
  which is made by approximating the empirical spectral process $\tr\left(E_{n}(g)\right)$ by the empirical spectral process $\tr(E_{n,N^{(\Delta)}}(g^\Phi))$ is sufficiently small.
 \begin{lemma}\label{ApproxPerWN}~\\
    Let Assumption $A$ hold. Let $\mathcal G_m\subseteq \mathcal H_m$  be totally bounded.
 \begin{itemize}
 	\item[(a)] Suppose $\sup_{g\in\mathcal G_m}\|g\|_2<\infty$.  Then,
 	$\|E_{n,R}\|_{\mathcal{G}_m}\overset{\P}{\longrightarrow}0$ as $n\to\infty$.
    \item[(b)] Suppose $\mathcal G_m$ is a Hilbert space. Then, $\|E_{n,R}\|^{\Phi,s}_{\mathcal G_m^\prime}\overset{\P}{\longrightarrow}0$ as $n\to\infty$.
 \end{itemize}
 \end{lemma}
 \begin{proof} $\mbox{}$\\
(a) \,  Define $R_n(\omega)=I_{n,Y^{(\Delta)}}(\omega)-\Phi(e^{-i\omega})I_{n,N^{(\Delta)}}(\omega) \Phi(e^{i\omega})^\top$ for $\omega \in [-\pi,\pi]$.
Due to \Cref{SSPsampledNat} we get
\begin{align*}
 	R_n(\omega)
 	&=\frac{1}{2\pi n}\left(\sum_{k=1}^{n}\sum_{p=0}^{\infty} \Phi_p N_{k-p}\D \right)\left(\sum_{\ell=1}^{n}\sum_{t=0}^{\infty}\Phi_tN_{\ell-t}\D \right)^\top e^{-i(k-\ell)\omega}\\
 	&\quad \left. -\frac{1}{2\pi n}\sum_{k=1}^{n}\sum_{p=0}^{\infty}\Phi_pN_k\D \right)\left(\sum_{\ell=1}^{n}\sum_{t=0}^{\infty}\Phi_tN_\ell\D \right)^\top e^{-i(k+p-\ell-t) \omega} \\
 	&=\frac{1}{2\pi n} \sum_{p=0}^{\infty}\sum_{t=0}^{\infty} \Phi_p  \left( \sum_{k=1-p}^{0}\sum_{\ell=1-t}^{0} -\sum_{k=1}^{n}\sum_{\ell=n-t+1}^{n}+\sum_{k=1-p}^{0}\sum_{\ell=1}^{n}-\sum_{k=1-p}^{0}\sum_{\ell=n-t+1}^{n} \right.   \\
 	& \quad \quad \quad  \  \left. +\sum_{k=n-p+1}^{n}\sum_{\ell=n-t+1}^{n}- \sum_{k=n-p+1}^{n}\sum_{\ell=1}^{n}+\sum_{k=1}^{n}\sum_{\ell=1-t}^{0}-\sum_{k=n-p+1}^{n}\sum_{\ell=1-t}^{0}\right)\\
 	&\quad\quad \quad \quad  \ \ \  N_k\D N_\ell^{(\Delta)\top}e^{-i\omega(k+p-\ell-t)}\Phi_t^\top \\
 	&=:\sum_{i=1}^{8}R_n^{(i)}(\omega).
 	\end{align*}
 	Thus, we show that
 	\beao
 	\sup_{\substack{g\in \mathcal G_{m}}}\left\| \sqrt n \int_{-\pi}^{\pi}g(\omega)R_n^{(i)}(\omega)d\omega\right\| \overset{\P}{\longrightarrow}0, \quad i=1,\ldots,8,
 	\eeao
 	holds.
 	By symmetry, the proofs for $ i=6, 7,8$ are the same as those for $i= 2, 3, 4,$ respectively.
 	Note that the proofs for $i=4,5$ are based on the same ideas as the proof for $i=1$; the case $i=3$ goes very similar to the case  $i=2$.  Consequently, we only investigate the terms corresponding to $i=1,2$.

 	As a first step we consider the case  \underline{$i=1$}:
 \begin{eqnarray} \label{eqref14}
 		\lefteqn{\sup_{g\in \mathcal G_{m}}\left\| n \int_{-\pi}^{\pi}g(\omega)R_n^{(1)}(\omega)d\omega\right\|}\nonumber\\
 		&=&\sup_{g\in \mathcal G_m} \left\|\sum_{p=0}^{\infty}\sum_{t=0}^\infty \sum_{\ell =1-t}^{0}\sum_{u=-\ell-t+1}^{-\ell-t+p}  \widehat g_u\Phi_p N_{\ell+t+u-p}^{(\Delta)}N_\ell^{(\Delta)\top}\Phi_t^\top \right\|  \nonumber\\
 		&\leq& \sup_{g\in \mathcal G_m}\sum_{p=0}^{\infty}\sum_{t=0}^\infty \sum_{\ell =1-t}^{0} \sum_{u=1}^{p} \|\Phi_p\| \|N_{u-p}^{(\Delta)}\| \|N_\ell^{(\Delta)}\| \|\Phi_t\|  \|\widehat g_{u-\ell-t}\|  \nonumber\\
 		&\leq& \sup_{g\in \mathcal G_m} \sum_{p=0}^{\infty}\sum_{t=0}^\infty \sum_{u=1}^{p} \|\Phi_p\|  \|\Phi_t\|\|N_{u-p}^{(\Delta)}\| \sqrt{\left(\sum_{\ell =1-t}^{0}  \|N_{\ell}^{(\Delta)}\|^2\right) \left(\sum_{\ell =1-t}^{0}  \|\widehat g_{u-\ell-t}\|^2 \right) }  \nonumber\\
 		&\leq& \mathfrak C\sum_{p=0}^{\infty} \|\Phi_p\| (p+1) \sum_{t=0}^\infty \sqrt {t+1} \|\Phi_t\| \left(\frac 1 {1+p}\sum_{u=1-p}^{0}  \|N_{u}^{(\Delta)}\|\right)\sqrt{\frac 1 {t+1}\sum_{\ell =1-t}^{0}  \|N_{\ell}^{(\Delta)}\|^2 }, \nonumber\\
 	\end{eqnarray}
 	where we used that $\sup_{\substack{g\in \mathcal G_m}}\|g\|_2^2 <\infty$. Note that $\sum_{p=0}^\infty (1+p)\|\Phi_p\|$ and $\sum_{t=0}^\infty \sqrt {t+1} \|\Phi_t\|$
 are finite due to \eqref{PhiBound}, and that due to the strong law of large numbers $$\frac{1}{1+p}\sum_{u=1-p}^0 \|N_u^{(\Delta)}\| \overset{a.s.}{\to} \E\|N_1^{(\Delta)}\|,\quad \text{ as } p\to\infty.$$
 	Thus, the right hand side of \eqref{eqref14} is a.s. finite
    and
 $$\sup_{g\in \mathcal G_m} \left\| \sqrt n \int_{-\pi}^{\pi}g(\omega)R_n^{(1)}(\omega)d\omega\right\|\overset{\P}{\longrightarrow}0,\quad \text{ as } n\to \infty.$$

 	Next, we investigate \underline{$i=2$}:
 	Since the space $(\mathcal G_m,d_2)$ is totally bounded, we can approximate the supremum over the potentially uncountable many functions in $\mathcal G_m$ by
  a maximum of finitely many functions. Namely, let $\delta>0$. Then, there exist a $v \in \N$ and $g_1,\ldots,g_v \in \mathcal G_m$ such that \linebreak $\sup_{g\in \mathcal G_m} \min_{j=1,\ldots,v}d_2(g,g_j)<\delta.$ Therefore, for fixed $\delta>0$ and appropriately chosen $g_1,\ldots, g_v$, we can approximate the error term by
 \begin{eqnarray*}
 		\lefteqn{\sup_{g\in \mathcal G_m} \left\|\sqrt n \int_{-\pi}^{\pi}g(\omega)R_n^{(2)}(\omega)d\omega\right\|}\\
 		&\leq&\sup_{g\in \mathcal G_m} \max_{j=1,\ldots,v \atop d_2(g,g_j)<\delta} \left\|\sqrt n \int_{-\pi}^{\pi}(g(\omega)-g_j(\omega))R_n^{(2)}(\omega)d\omega\right\|+\max_{j=1,\ldots,v} \left\|\sqrt n \int_{-\pi}^{\pi}g_j(\omega)R_n^{(2)}(\omega)d\omega\right\|\\
 		&\leq & \sqrt{2\pi} \delta \left( n \int_{-\pi}^{\pi} \left\|R_n^{(2)}(\omega)\right\|^2 d\omega\right)^{1/2} +\max_{j=1,\ldots,v} \left\|\sqrt n \int_{-\pi}^{\pi}g_j(\omega)R_n^{(2)}(\omega)d\omega\right\|.
 	\end{eqnarray*}
 	Since $\delta$ can be chosen arbitrary small, it is sufficient to prove
 \begin{eqnarray}
 	n \int_{-\pi}^\pi \|R_n^{(2)}(\omega)\|^2d\omega  &=&O_{\P}(1), \quad \quad\text{ and } \label{ex1}\\
  \left\|\sqrt n \int_{-\pi}^{\pi}g(\omega)R_n^{(2)}(\omega)d\omega\right\|&=&o_\P(1) \quad\quad\quad \text{for }  g\in \mathcal G_m. \label{exx1}
 	\end{eqnarray}
 	On the one hand, we have
 	\begin{eqnarray*}
 		\lefteqn{\E\left[ n \int_{-\pi}^\pi \|R_n^{(2)}(\omega)\|^2d\omega\right]}\\
 		&=& \sum_{S,T=1}^{r} \int_{-\pi}^{\pi}\frac 1 {4\pi^2n}\sum_{p_1,p_2=0}^{\infty}\sum_{t_1,t_2=0}^\infty \sum_{k_1,k_2=1}^n \sum_{\ell_1=n+1-t_1}^{n}\sum_{\ell_2=n+1-t_2}^n \E\left[\left(\Phi_{p_1}N_{k_1}^{(\Delta)}N_{\ell_1}^{(\Delta)\top} \Phi_{t_1}^\top \right)[S,T]\right. \\
 		&&\left.\qquad \qquad\times\left( \Phi_{p_2} N_{k_2}^{(\Delta)}N_{\ell_2}^{(\Delta)\top}\Phi_{t_2}^\top \right)[S,T]\right]e^{i\omega(k_1+p_1-\ell_1-t_1-k_2-p_2+\ell_2+t_2)}d\omega  \\
 		&=& \sum_{S,T=1}^{r} \int_{-\pi}^{\pi}\frac 1 {4\pi^2n}\sum_{p_1,p_2=0}^{\infty}\sum_{t_1,t_2=0}^\infty \sum_{k=1}^n \sum_{\ell=n+1-\min\{t_1,t_2\}}^{n}\E\left[\left( \Phi_{p_1}N_{k}^{(\Delta)}N_{\ell}^{(\Delta)\top}\Phi_{t_1}^\top\right) [S,T]\right. \\
 		&&\left. \qquad \qquad  \times\left(\Phi_{p_2} N_{k}^{(\Delta)}N_{\ell}^{(\Delta)\top}\Phi_{t_2}^\top \right)[S,T]\right]e^{i\omega(p_1-t_1-p_2+t_2)}d\omega  \\
 		&&+  \sum_{S,T=1}^{r} \int_{-\pi}^{\pi}\frac 1 {4\pi^2n}\sum_{p_1,p_2=0}^{\infty}\sum_{t_1,t_2=0}^\infty \sum_{k=\max\{1,n+1-t_1\}}^n \sum_{\substack{\ell=\max\{1,n+1-t_2\},\\\ell\neq k}}^n \left(\Phi_{p_1}\E\left[N_{1}^{(\Delta)}N_{1}^{(\Delta)\top}\right]\Phi_{t_1}^\top\right) [S,T] \\
 		&& \qquad \qquad  \times\left(\Phi_{p_2} \E\left[N_{1}^{(\Delta)}N_{1}^{(\Delta)\top}\right]\Phi_{t_2}^\top\right) [S,T]e^{i\omega(p_1-t_1-p_2+t_2)}d\omega  \\
 		&&+  \sum_{S,T=1}^{r} \int_{-\pi}^{\pi}\frac 1 {4\pi^2n}\sum_{p_1,p_2=0}^{\infty}\sum_{t_1,t_2=0}^\infty  \sum_{k=\max\{1,n+1-t_2\}}^n \sum_{\substack{\ell=\max\{1,n+1-t_1\},\\\ell\neq k}}^n \E\left[\left(\Phi_{p_1}N_{k}^{(\Delta)}N_{\ell}^{(\Delta)\top}\Phi_{t_1}^\top\right) [S,T]\right. \\
 		&&\left. \qquad \qquad \times\left(\Phi_{p_2} N_{\ell}^{(\Delta)}N_{k}^{(\Delta)\top}\Phi_{t_2}^\top \right)[S,T]\right]e^{i\omega(2k-2\ell+p_1-t_1-p_2+t_2)}d\omega  \\
 		&\leq&\frac{r^2}{2\pi n} \left(\sum_{p=0}^\infty \|\Phi_{p}\| \right)^2 \sum_{t_1,t_2=0}^\infty\|\Phi_{t_1}\| \|\Phi_{t_2}\|  \left(\sum_{\substack{\ell=n+1-\min\{t_1,t_2\}}}^{n} \sum_{k=1}^n +2 \sum_{\substack{k=\max\{1,n+1-t_1\}\\\ell=\max\{1,n+1-t_2\}\\\ell\neq k}}^n\right) \E\left[\|N_1^{(\Delta)}\|^2\right]^2.
 	\end{eqnarray*}
 But this term is uniformly bounded due to \eqref{PhiBound} which results in \eqref{ex1}.

 	On the other hand, we have
 \begin{eqnarray}\nonumber\lefteqn{\left\|\sqrt n \int_{-\pi}^{\pi}g(\omega)R_n^{(2)}(\omega)d\omega\right\|} \\
 	\nonumber&\leq& \left\| \frac 1 {2\pi \sqrt n}\int_{-\pi}^\pi g(\omega)\sum_{t=0}^{n} \sum_{k=1}^{n-t} \sum_{\ell=1+n-t}^n \Phi(e^{-i\omega})N_k^{(\Delta)}N_\ell^{(\Delta)\top}\Phi_t^\top e^{-i\omega(k-\ell-t)}d\omega\right\| \\
 	\nonumber&&+\left\| \frac 1 {2\pi \sqrt n}\int_{-\pi}^\pi g(\omega) \sum_{t=0}^{n} \sum_{k=n+1-t}^{n} \sum_{\ell=1+n-t}^n \Phi(e^{-i\omega})N_k^{(\Delta)}N_\ell^{(\Delta)\top}\Phi_t^\top e^{-i\omega(k-\ell-t)}d\omega\right\|\\
 	&&+\left\| \frac 1 {2\pi \sqrt n}\int_{-\pi}^\pi g(\omega) \sum_{t=n+1}^{\infty} \sum_{k=1}^n \sum_{\ell=1+n-t}^n \Phi(e^{-i\omega})N_k^{(\Delta)}N_\ell^{(\Delta)\top}\Phi_t^\top e^{-i\omega(k-\ell-t)}d\omega\right\|\nonumber\\
 	&=:& R_{n,1}^{(2)}+R_{n,2}^{(2)}+R_{n,3}^{(2)}.\label{Gl123}
 	\end{eqnarray}
 	We investigate the terms in \eqref{Gl123} separately.
 	For the first one, the independency of the sequence $N^{(\Delta)}$, $k\not=l$ and the Cauchy-Schwarz inequality  yield
 	\begin{eqnarray}
 	\nonumber\lefteqn{\E\left[\left(R_{n,1}^{(2)}\right)^2\right]}\\
 	\nonumber&=&\frac{1}{4\pi^2 n}\sum_{k=1}^{n-1}\sum_{\ell=1+k}^n\E\left[\left\|\int_{-\pi}^\pi \sum_{t=1+n-\ell}^{n-k}g(\omega)\Phi(e^{-i\omega})N_k^{(\Delta)}N_\ell^{(\Delta)\top}\Phi_t^{\top}e^{-i\omega(k-\ell-t)}d\omega\right\|^2\right]\\
 	\nonumber&\leq& \frac{\mathfrak C}{n}\sum_{k=1}^{n}\sum_{\ell=1+k}^n \left(\sum_{t=1+n-\ell}^{n-k}\|\Phi_{t}\|^2t^2\right)\left(\sum_{t=1+n-\ell}^{n-k}t^{-2}\left\| \int_{-\pi}^{\pi}g(\omega)\Phi(e^{-i\omega})e^{-i\omega(k-\ell-t)}d\omega\right\|^2\right)\\
\nonumber&\leq& \frac{\mathfrak C}{n}\sum_{k=1}^{n} \left(\sum_{t=1}^{n-k}\sum_{\ell=1+n-t}^n\|\Phi_{t}\|^2t^2\right)\left(\sum_{t=1}^{\infty}t^{-2}\left\| \int_{-\pi}^{\pi}g(\omega)\Phi(e^{-i\omega})e^{-i\omega(k-t-n)}d\omega\right\|^2\right)\\
\nonumber&\leq & \frac{\mathfrak C}{n}\left(\sum_{t=0}^\infty \|\Phi_t\|^2 t^3 \right)\sum_{t=1}^{\infty}\frac 1 {t^2}\sum_{k=-\infty}^\infty\left\| \reallywidehat{\left(g(\cdot)\Phi(e^{-i\cdot})\right)}_k\right\|^2.
 	\end{eqnarray}
Finally, due to \eqref{PhiBound} and \eqref{FK_absch} we receive
$$\E\left[\left(R_{n,1}^{(2)}\right)^2\right]\leq \frac{\mathfrak C}{n}\overset{n \to \infty}{\longrightarrow}0.$$
  	Next, we investigate the second term in \eqref{Gl123}. By the independence of the sequence $N^{(\Delta)}$ similar calculations as above
    give
 	\begin{eqnarray}
 	\nonumber\lefteqn{\E\left[\left(R_{n,2}^{(2)}\right)^2\right]}\\
 	\nonumber&\leq &\frac{1}{4\pi^2 n}\E\left[\left\|\sum_{k=1}^{n}\sum_{\ell=1}^n\sum_{t=1+n-\min\{k,\ell\}}^{n}\int_{-\pi}^\pi g(\omega)\Phi(e^{-i\omega})N_k^{(\Delta)}N_\ell^{(\Delta)\top}\Phi_t^{\top}e^{-i\omega(k-\ell-t)}d\omega\right\|^2\right]\\
 	\nonumber&\leq& \frac{1}{4\pi^2 n}\sum_{k=1}^{n}\sum_{\ell=1}^n\E\left[\left\|\sum_{t=1+n-\min\{k,\ell\}}^{n}\int_{-\pi}^\pi g(\omega)\Phi(e^{-i\omega})N_k^{(\Delta)}N_\ell^{(\Delta)\top}\Phi_t^{\top}e^{-i\omega(k-\ell-t)}d\omega\right\|^2\right]\\
 	\nonumber&&\quad+\frac{1}{4\pi^2 n}\left\|\sum_{k=1}^{n}\sum_{t=1+n-k}^{n}\int_{-\pi}^\pi g(\omega)\Phi(e^{-i\omega})\Sigma_N^{(\Delta)}\Phi_t^{\top}e^{i\omega t}d\omega\right\|^2\\
 	\nonumber&&\quad+\sum_{S,T=1}^r\frac{1}{4\pi^2 n}\sum_{k=1}^n\sum_{\substack{\ell=1\\\ell\neq k}}^n\sum_{\substack{\substack{t_1,t_2=1+n -\min\{k,\ell\}}}}^n\E\left[\int_{-\pi}^\pi \left[g(\omega)\Phi(e^{-i\omega})N_k^{(\Delta)}N_\ell^{(\Delta)\top}\Phi_{t_1}^{\top}\right][S,T]e^{-i\omega(k-\ell-t_1)}d\omega\right.\\
 	\nonumber&&\qquad \qquad \qquad\qquad\qquad \qquad  \left.\times\int_{-\pi}^\pi \left[g(\omega)\Phi(e^{i\omega})N_\ell^{(\Delta)}N_k^{(\Delta)\top}\Phi_{t_2}^{\top}\right][S,T]e^{-i\omega(k-\ell+t_2)}d\omega\right].
 \end{eqnarray}
 Furthermore, Cauchy-Schwarz inequality gives the upper bound
 \begin{eqnarray}
 	\nonumber&\leq& \frac{\mathfrak C}{ n}\sum_{k=1}^{n}\sum_{\ell=1}^n\left(\sum_{\substack{t=1+n-\min\{k,\ell\}}}^{n}\frac 1 {t^2}\left\|\int_{-\pi}^\pi g(\omega)\Phi(e^{-i\omega})e^{-i\omega(k-\ell-t)}d\omega\right\|^2\right) \left(\sum_{\substack{t=1+n-\min\{k,\ell\}}}^n\|\Phi_t\|^2 t^2\right)\\
 	\nonumber&&\quad+\frac{\mathfrak C}{n}\left(\sum_{t=1}^n t^2 \|\Phi_t\|^2\right)\left(\sum_{t=1}^{n}\left\|\int_{-\pi}^\pi g(\omega)\Phi(e^{-i\omega})e^{i\omega t}d\omega\right\|^2\right)
            \\
 	\nonumber&&\quad+\frac{\mathfrak C}{ n}\sum_{k=1}^n\sum_{\substack{\ell=1\\\ell\neq k}}^n\sum_{\substack{t_1,t_2=1+n -\min\{k,\ell\}}}^n\left\|\int_{-\pi}^\pi g(\omega)\Phi(e^{-i\omega})e^{-i\omega(k-\ell-t_1)} d\omega\right\| \|\Phi_{t_1}\|\\
 	\nonumber&&\qquad \qquad \qquad\qquad\quad \qquad \times \left\|\int_{-\pi}^\pi g(\omega)\Phi(e^{i\omega})e^{-i\omega(k-\ell+t_2)}d\omega\right\| \|\Phi_{t_2}\| \\
 		\nonumber&\leq&\frac{\mathfrak C}{ n}\sum_{\ell=1}^n\left(\sum_{t=1+n}^{n+\ell}\frac 1 {(t-\ell)^2}\left(\sum_{k=-\infty}^{\infty}\left\|\reallywidehat{\left(g(\cdot)\Phi(e^{-i\cdot})\right)}_{k}\right\|^2\right)\right) \left(\sum_{t=1+n-\ell}^n\|\Phi_t\|^2 t^2\right) \\
 	\nonumber&&\quad+\frac{\mathfrak C}{n}\left(\sum_{t=1}^n t^2 \|\Phi_t\|^2\right)\left(\sum_{t=1}^{n}\left\|\reallywidehat{g(\cdot)\Phi(e^{-i\cdot})}_{t}\right\|^2\right)\\
 	\nonumber&&\quad+\frac{\mathfrak C}{ n}\max_{t\in \Z}\left\|\reallywidehat{\left(g(\cdot)\Phi(e^{-i\cdot})\right)}_{t}\right\|\max_{t\in \Z}\left\|\reallywidehat{\left(g(\cdot)\Phi(e^{i\cdot})\right)}_{t}\right\|\sum_{t_1,t_2=1}^nt_1 \|\Phi_{t_1}\|t_2 \|\Phi_{t_2}\|.
 	\end{eqnarray}
 A conclusion of \eqref{PhiBound} and \eqref{FK_absch} is then $
\E\left[\left(R_{n,2}^{(2)}\right)^2\right]\leq \mathfrak C/n\overset{n\to \infty}{\longrightarrow}0.$

 	The convergence of $R_{n,3}^{(2)}$ can be proven similarly.

(b) \, The proof goes analogue to (a)  by replacing $\sup_{g\in \mathcal G_m }$ by $\sup_{g\in \mathcal G_m \atop
\|g\|_{\Phi,s}\leq 1}$. Therefore, we have to take \Cref{corollary:aux} into account implying that
there exists a constant $\mathfrak C>0$ such that $\sup_{g\in \mathcal G_m \atop
\|g\|_{\Phi,s}\leq 1}\|g\|_2\leq \mathfrak C$.
 \end{proof}

 \noindent \textit{Proof of \Cref{Satz::intP::Y2}}~\\
 We use the decomposition $\tr(E_{n}(g))=\tr(E_{n,N^{(\Delta)}}(g^\Phi))+E_{n,R}(g)$.
 A conclusion of \Cref{ApproxPerWN} (a) is that $\|E_{n,R}\|_{\mathcal{G}_m}\overset{\P}\longrightarrow0$ as $n\to\infty$
 and hence, it has no influence on the asymptotic behavior of $\tr(E_n(g))$.

 It remains to investigate $\tr(E_{n,N^{(\Delta)}}(g^\Phi))$. Therefore,
 define $\mathcal G_r:=\{g^\Phi: g \in \mathcal G_m\}\subseteq \mathcal H_r.$ Due to the continuous mapping theorem, it
 is sufficient to show that $(E_{n,N^{(\Delta)}}(g^\Phi))_{g^\Phi\in \mathcal{G}_r}$ converges weakly in $(\mathcal C(\mathcal G_r,\C^{r\times r}),d_{\mathcal{G}_r})$.
 Indeed, Assumption B for $\mathcal G_m$ implies that the set $\mathcal{G}_r$ satisfies the analogue conditions in Assumption N (see  \Cref{Remark 3.1}),
 and as well $\sup_{g^{\Phi}\in\mathcal{G}_r}\|g^\Phi\|_{*,s}=\sup_{g\in\mathcal{G}_m}\|g\|_{\Phi,s}<\infty$ .
 Consequently, an application of \Cref{Theorem:5.2} yields
 $$(E_{n,N^{(\Delta)}}(g^\Phi))_{g^{\Phi}\in \mathcal G_r}\overset{\mathcal D}{\longrightarrow}
        (E(g^\Phi))_{g^\Phi\in\mathcal G_{r}} \quad\text{ in }(\mathcal C (\mathcal G_{r},\C^{r\times r}),d_{\mathcal G_{r}}),$$
 such that we are able to conclude the statement.
 \hfill$\Box$

 \subsection{Proof of \Cref{Satz::intP::Y}}

 \noindent \textit{Proof of \Cref{Satz::intP::Y}}~\\
 The proof goes as the proof of  \Cref{Satz::intP::Y2} using \Cref{Satz::intP::N} instead of \Cref{Theorem:5.2} and \Cref{ApproxPerWN} (b) instead of \Cref{ApproxPerWN}(a).
 \hfill$\Box$

\subsection{Proof of \Cref{Theorem:main3}} \label{Section:6.3}

\noindent\textit{Proof of \Cref{Theorem:main3}}~\\
Define the Hilbert space $\mathcal G_r:=\{g^{\Phi}: g\in \mathcal H_m^s\}\subseteq \mathcal H_r$ with scalar product $$\langle g^\Phi,f^\Phi\rangle_{\mathcal G_r} =\sum_{h=-\infty}^{\infty}(1+|h|)^{2s}\tr\left(\widehat {g_h^\Phi} (\widehat {f_h^\Phi})^\top\right)$$ for $g^\Phi,f^\Phi \in \mathcal G_r$.
We use as well the representation $\tr\left(E_{n}(g)\right)=\tr(E_{n,N^{(\Delta)}}(g^\Phi))+E_{n,R}(g).$
Since Assumption (B1) is satisfied, we have by \Cref{Satz::intP::N} as $n\to \infty$,
 $$(\tr(E_{n,N^{(\Delta)}}(g^\Phi)))_{g^\Phi\in \mathcal G_r}\overset{\mathcal D}{\longrightarrow}
        (\tr(E(g^\Phi)))_{g^\Phi\in \mathcal G_{r}} \quad \text{ in } ({{\mathcal G_r^*}^\prime},d_{{\mathcal G^{\ast\prime}_r}}^s).$$
        If we show that $\sup_{\substack{g\in \mathcal H_{m}^s} \atop \|g\|_{\Phi,s}\leq 1}\left\|E_{n,R}(g)\right\|\overset{\P}{\longrightarrow}0$
        the same arguments as in the proof of \Cref{Satz::intP::Y2} finish the proof.
Indeed, as in the proof of \Cref{ApproxPerWN} we will show that
 	\beao
 	\sup_{\substack{g\in \mathcal H_{m}^s} \atop \|g\|_{\Phi,s}\leq 1}\left\| \sqrt n \int_{-\pi}^{\pi}g(\omega)R_n^{(i)}(\omega)d\omega\right\| \overset{\P}{\longrightarrow}0, \quad i=1,\ldots,8,
 	\eeao
 and therefore, it is sufficient to investigate the terms  $i=1,2$.

The term \underline{$i=1$} can be handled as in \Cref{ApproxPerWN} because $\sup_{g\in \mathcal H_m^s  \atop \|g\|_{\Phi,s}\leq 1}\|g\|_2<\infty$ due to \Cref{corollary:aux}.
 	Next, we investigate \underline{$i=2$}:
The Cauchy-Schwarz inequality yields
 	\begin{eqnarray*}
 		\lefteqn{ 	\sup_{\substack{g\in \mathcal H_{m}^s} \atop \|g\|_{\Phi,s}\leq 1}\left\|\sqrt n \int_{-\pi}^{\pi}g(\omega)R_n^{(2)}(\omega)d\omega\right\|}\\
 		&=&	\sup_{\substack{g\in \mathcal H_{m}^s} \atop \|g\|_{\Phi,s}\leq 1} \left\|\frac 1 {\sqrt n}\int_{-\pi}^{\pi} \frac 1 {2\pi} \sum_{p=0}^{\infty}\sum_{t=0}^\infty \sum_{k=1}^{n}\sum_{\ell =n+1-t}^{n}\sum_{u=-\infty}^{\infty} \widehat g_u\Phi_p N_k^{(\Delta)}N_\ell^{(\Delta)\top}\Phi_t^\top  e^{-i\omega(k+p-\ell-t-u)}d\omega\right\| \\
&\leq&	 	\sup_{\substack{g\in \mathcal H_{m}^s} \atop \|g\|_{\Phi,s}\leq 1}\frac 1 { \sqrt n} \sum_{t=0}^\infty \left\|\Phi_t\right\|\sum_{p=0}^{\infty}\|\Phi_p\| \sum_{\ell=n+1-t}^n\left\|N_\ell^{(\Delta)}\right\| \sum_{u=1+p-\ell-t}^{n+p-\ell-t} \left\| \widehat g_{u}\right\|  \left\| N_{\ell+u+t-p}^{(\Delta)}\right\|  \\
	&\leq&	 	\sup_{\substack{g\in \mathcal H_{m}^s} \atop \|g\|_{\Phi,s}\leq 1}\frac 1 { \sqrt n} \sum_{t=0}^\infty \left\|\Phi_t\right\|\sum_{p=0}^{\infty}\|\Phi_p\| \sum_{\ell=n+1-t}^n\left\|N_\ell^{(\Delta)}\right\|\left( \sum_{u=1+p-\ell-t}^{n+p-\ell-t} (1+|u|)^{2s} \left\| \widehat g_{u}\right\|^2\right)^{1/2} \\&&\qquad \qquad \qquad\qquad\qquad \qquad\qquad \qquad\qquad \times \left( \sum_{u=1-\ell-t}^{n-\ell-t}(1+|u|)^{-2s}\left\| N_{\ell+u+t-p}^{(\Delta)}\right\|^2\right)^{1/2} .
 	\end{eqnarray*}
By an application of \Cref{lemma:aux1}(a) and $\|g\|_{\Phi,s}\leq 1$, we obtain $	\sup_{\substack{g\in \mathcal H_{m}^s} \atop \|g\|_{\Phi,s}\leq 1} \sum_{u=-\infty}^\infty (1+|u|)^{2s}\|\widehat g_{u}\|^2 \leq \mathfrak{C}$ and hence, the upper bound
 \begin{eqnarray}
\nonumber 	\lefteqn{ \|\sqrt n \int_{-\pi}^{\pi}g(\omega) R_n^{(2)}(\omega)d\omega \|^{\Phi,s}_{\mathcal H_m^{s\prime}}}\\
 	&\leq&\frac{\mathfrak C} { \sqrt n} \sum_{t=0}^\infty \left\|\Phi_t\right\|\sum_{p=0}^{\infty}\|\Phi_p\| \sum_{\ell=n+1-t}^n\left\|N_\ell^{(\Delta)}\right\| \max_{j=1,\ldots,n}\left\| N_{j}^{(\Delta)}\right\|\left( \sum_{u=-\infty}^{\infty}(1+|u|)^{-2s}\right)^{1/2}\label{upperR2}
 	\end{eqnarray}
 	follows. Note that $n^{-1/2}\max_{j=1,\ldots,n}\|N_j^{(\Delta)}\|\overset{\P}{\rightarrow}0$ as $n\to\infty$, since for any $\varepsilon>0$:	\begin{align*}
 \P\left(\frac 1 {\sqrt n} \max_{j=1,\ldots,n}\|N_j^{(\Delta)}\|>\varepsilon\right)
 &\leq 1-\left(1-\frac {\E[\|N_1^{(\Delta)}\|^4]}{\varepsilon^4 n^2}\right)^n \overset{ n\to \infty}{\longrightarrow} 0.
 \end{align*}
 Then,   $\sum_{u=-\infty}^{\infty} (1+|u|)^{-2s}<\infty $ for $s>1/2$, $ \sum_{t=0}^\infty t\|\Phi_t\|<\infty$, $\frac 1 t \sum_{\ell=n+1-t}^{n}\|N_\ell^{(\Delta)}\|=O_\P(1)$ and \eqref{upperR2} yield
  \begin{eqnarray*}
  	\left\|\sqrt n \int_{-\pi}^{\pi}g(\omega) R_n^{(2)}(\omega)d\omega \right\|_{\mathcal H_m^{s\prime}}^{\Phi,s}\overset{\P}{\longrightarrow}0.
  	\end{eqnarray*}
  The remaining terms $i=3,\ldots,8$ can be handled similarly.
\hfill$\Box$

\begin{appendix}

\section{Appendix}

\subsection{Proof of \Cref{lemma:aux1}}

\noindent\textit{Proof of \Cref{lemma:aux1}.}
\begin{itemize}
	\item[(a)]
Let $g\in  \mathcal H_m$. 
Note that the Fourier coefficients of $g^\Phi$ satisfy
\begin{eqnarray}
	\widehat {g^\Phi_h}&=&  \sum_{\substack{j_1,j_2=0}}^\infty \Phi_{j_1}^\top \widehat g_{h-j_1+j_2}\Phi_{j_2}= \sum_{\substack{j_1,j_2=0}}^\infty e^{A^\top \Delta j_1}C^\top \widehat g_{h-j_1+j_2}Ce^{A\Delta j_2}, \quad h\in\Z. \label{FC_rep}
\end{eqnarray}Therefore, we have the representation
\begin{eqnarray*}
	C^\top\widehat g_h C &=&\left[ \sum_{j_1,j_2=0}^\infty-\sum_{j_1=1}^\infty\sum_{j_2=0}^\infty-\sum_{j_1=0}^\infty\sum_{j_2=1}^\infty +\sum_{j_1,j_2=1}^\infty\right]e^{A^\top \Delta j_1} C^\top \widehat g_{h-j_1+j_2}Ce^{A\Delta j_2}   \\
	&=&\widehat{g^\Phi_h}-e^{A^\top \Delta}\widehat{g^\Phi_{h-1}}-\widehat{g^\Phi_{h+1}}e^{A\Delta}+e^{A^\top \Delta}\widehat{g^\Phi_{h}}e^{A\Delta}.
	\end{eqnarray*}
Thus,
\beao
            \|C^\top\widehat g_h C\|\leq \mathfrak C\left[\|\widehat{ g^\Phi_h}\|+\|\widehat{ g^\Phi_{h-1}}\|+\|\widehat{g^\Phi_{h+1}}\|\right].
        \eeao
Assumption A says that  $CC^\top=Id_m$, hence $\|C^\top \widehat g_h C\|=\|\widehat g_h\|$  and therefore, the assertion follows.
\item[(b)] Due to \eqref{FC_rep} it is sufficient to show that for some $\nu>0$ and any $h>0$ $$\sum_{j_1,j_2=0}^\infty \|\Phi_{j_1}\| \|\Phi_{j_2}\| e^{-\lambda|h-j_1+j_2|}\leq \mathfrak C e^{-\nu h}.$$
Let $h>0$. Since $\|\Phi_j\|\leq e^{-\mu j}$ for $j\in \N_0$ and some $\mu>0$, we obtain
\begin{eqnarray*}
	\lefteqn{\sum_{j_1,j_2=0}^\infty \|\Phi_{j_1}\| \|\Phi_{j_2}\| e^{-\lambda|h-j_1+j_2|}}\\
	&=&\sum_{j_2=0}^\infty \sum_{j_1=0}^{h+j_2}\|\Phi_{j_1}\| \|\Phi_{j_2} \| e^{-\lambda(h-j_1+j_2)}+\sum_{j_2=0}^\infty \sum_{j_1=h+j_2+1}^{\infty}\|\Phi_{j_1}\| \|\Phi_{j_2} \| e^{\lambda(h-j_1+j_2)} \\
	&\leq&\mathfrak C e^{-\lambda h}\sum_{j_2=0}^\infty e^{-(\mu+\lambda)j_2} \left(\frac{1-e^{(-\mu+\lambda)(h+j_2+1)}}{1-e^{-\mu+\lambda}}\right) +e^{\lambda h}\sum_{j_2=0}^\infty e^{(-\mu+\lambda)j_2}\left(\frac{e^{-(\mu+\lambda)(h+j_2+1)}}{1-e^{-\mu-\lambda}}\right) \\
	&\leq& \mathfrak C \left(e^{-\lambda h}+e^{-\mu h}\right).
	\end{eqnarray*}
	Setting $\nu=\min\{\lambda, \mu\}$ yields the assertion.
\end{itemize}
\hfill$\Box$

\subsection{Proof of \Cref{Satz::intP::Y1}}
 \noindent\textit{Proof of \Cref{Satz::intP::Y1}}
 \begin{itemize}
 	\item[(a)] Define
     $$\mathfrak g_j(\omega):=  g^\Phi_j(\omega)+g^\Phi_j(-\omega)^\top=\Phi(e^{i\omega})^\top (g_j(\omega)+g_j(-\omega)^\top) \Phi(e^{-i\omega}), \quad \omega\in[-\pi,\pi],$$ for $j=1,2$ with Fourier coefficients  $\widehat{ (\mathfrak g_j)}_{h}=\widehat{ ( g_j^\Phi)}_{h}+\widehat{ ( g_j^\Phi)}_{-h}^\top$ for $h\in\Z$.
 	First of all,
 	\begin{eqnarray*}
 		\lefteqn{\sum_{h=-\infty}^\infty \tr\left(\Sigma_N^{(\Delta)}\widehat{ (\mathfrak g_1)}_h^H\Sigma_N^{(\Delta)}\widehat{ (\mathfrak g_2)}_h\right)}\\
 		&=&\frac 1 {2\pi}\sum_{h=-\infty}^{\infty} \sum_{\ell=-\infty}^{\infty} \tr\left(\Sigma_N^{(\Delta)}\widehat{ (\mathfrak g_1)}_\ell^H\Sigma_N^{(\Delta)}\widehat{ (\mathfrak g_2)}_h\right)\int_{-\pi}^\pi e^{i(\ell-h)\omega}d\omega\\ 
 		&=&\frac{1}{2\pi}\cdot2\pi\cdot 2\pi\int_{-\pi}^\pi \tr\left(\frac {\Sigma_N^{(\Delta)}} {2\pi}\Phi(e^{i\omega})^\top (g_1(\omega)+g_1(-\omega)^\top)^H \Phi(e^{-i\omega}) \right.\\
 		&&\left.\quad\quad\quad\quad\quad\quad\quad\quad \times \frac{\Sigma_N^{(\Delta)}} {2\pi}\Phi(e^{i\omega})^\top (g_2(\omega)+g_2(-\omega)^\top) \Phi(e^{-i\omega})\right)d\omega\\
 		&=&2\pi\int_{-\pi}^\pi \tr\left(\f(\omega) (g_1(\omega)+g_1(-\omega)^\top)^H \f(\omega) (g_2(\omega)+g_2(-\omega)^\top)\right) d\omega. 
 	\end{eqnarray*}
 	Then, since $(W_h)_{h\in \N}$ is an  i.i.d. centered sequence and the sum is well-defined, we obtain
 	\begin{eqnarray*}
 		\lefteqn{\Cov\left(\tr\left(\sum_{h=1}^\infty W_h \widehat{ (\mathfrak g_1)}_{h} \right) ,\tr\left(\sum_{h=1}^\infty W_h \widehat{ (\mathfrak g_2)}_{h}\right) \right) }\\
 		&=&\sum_{h=1}^\infty \tr\left(\Sigma_N^{(\Delta)}\widehat{ (\mathfrak g_1)}_h^H\Sigma_N^{(\Delta)}\widehat{ (\mathfrak g_2)}_h\right)\\
 		&=&\frac{1}{2}\sum_{h=-\infty}^\infty \tr\left(\Sigma_N^{(\Delta)}\widehat{ (\mathfrak g_1)}_h^H\Sigma_N^{(\Delta)}\widehat{ (\mathfrak g_2)}_h\right)-2\vecc\left( \widehat{\left(g_1^\Phi\right)}_0^\top\right)^\top \left(\Sigma_N^{(\Delta)}\otimes \Sigma_N^{(\Delta)}\right)\vecc\left( \widehat{\left(g_2^\Phi\right)}_0^H\right).
 	\end{eqnarray*}
 	Finally, since $W_0'$ is independent from $(W_h)_{h\in \N}$ we receive
 	\begin{eqnarray*}
 		\lefteqn{\Cov(\tr(E(g_1)),\tr(E(g_2)))}	\\
 		&=&\frac{1}{2}2\pi\int_{-\pi}^\pi \tr\left(\f(\omega) (g_1(\omega)+g_1(-\omega)^\top)^H \f(\omega) (g_2(\omega)+g_2(-\omega)^\top)\right) d\omega \\
 		&&-2\vecc\left( \widehat{\left(g_1^\Phi\right)}_0^\top\right)^\top \left(\Sigma_N^{(\Delta)}\otimes \Sigma_N^{(\Delta)}\right)\vecc\left( \widehat{\left(g_2^\Phi\right)}_0^H\right)\\
 		&&+\vecc\left( \widehat{\left(g_1^\Phi\right)}_0^\top\right)^\top \left(\E[N^{(\Delta)}_1N^{(\Delta)\top}_1 \otimes N^{(\Delta)}_1 N^{(\Delta)\top}_1]-\Sigma_N^{(\Delta)}\otimes \Sigma_N^{(\Delta)}\right)\vecc\left( \widehat{\left(g_2^\Phi\right)}_0^H\right)\\
 		&=&\pi\int_{-\pi}^\pi \tr\left(\f(\omega) (g_1(\omega)+g_1(-\omega)^\top)^H \f(\omega) (g_2(\omega)+g_2(-\omega)^\top)\right) d\omega \\
 		&&+\vecc\left( \widehat{\left(g_1^\Phi\right)}_0\right)^\top \left(\E[N^{(\Delta)}_1N^{(\Delta)\top}_1 \otimes N^{(\Delta)}_1 N^{(\Delta)\top}_1]-3\Sigma_N^{(\Delta)}\otimes \Sigma_N^{(\Delta)}\right)\vecc\left( \widehat{\left(g_2^\Phi\right)}_0^H\right).
 	\end{eqnarray*}
\item[(b)] Similar calculations as in (a) yield that the covariance functions coincide.
 	\item[(c)] Follows directly from (a) since for a Gaussian random vector $N^{(\Delta)}_1$ we have
 	$\E[N^{(\Delta)}_1N^{(\Delta)\top}_1 \otimes N^{(\Delta)}_1 N^{(\Delta)\top}_1]=3\Sigma_N^{(\Delta)}\otimes \Sigma_N^{(\Delta)}$.
 	\hfill$\Box$
 \end{itemize}

\subsection{Proof of \Cref{aux2}}

We start with an auxiliary result.

 \begin{lemma} \label{aux3}
 Let Assumption A and  $\sup_{g\in \mathcal G_{r}}\|g\|_{2}<\infty$ hold.
 \begin{itemize}
    \item[(a)] $
    \sup_{j\in \Z}n \E\left[\left\|\overline \Gamma_{n,N^{(\Delta)}}(j)-\E\left[\overline \Gamma_{n,N^{(\Delta)}}(j)\right]\right\|^2\right]\leq \mathfrak C.$
    \item[(b)] Define for any $g\in \mathcal G_r$, $M\in \N$,
    $$\widetilde E_{n,N^{(\Delta)}}^{(M)}(g)=\sqrt n \sum_{h=-M}^{M}\widehat g_h\left(\overline \Gamma_{n,N^{(\Delta)}}(h)-\E[\overline\Gamma_{n,N^{(\Delta)}}(h)]\right).$$
    Then, there exists a constant $K>0$ such that for any $M\in \N$: $$\E\left(\sup_{g\in \mathcal G_r}\|\widetilde E_{n,N^{(\Delta)}}^{(M)}(g)\|^2\right)<KM.$$
 \end{itemize}
 \end{lemma}
\begin{proof} $\mbox{}$
\begin{itemize}
	\item[(a)] On the one hand, we have
\begin{eqnarray*}
     n\E\|\overline\Gamma_{n,N^{(\Delta)}}(0)-\E[\overline\Gamma_{n,N^{(\Delta)}}(0)]\|^2
&=&\sum_{S,T=1}^{r}\Var\left(\left(N_1^{(\Delta)}N_1^{(\Delta)\top}\right)[S,T]\right).
\end{eqnarray*}
Since $\E\left[\overline \Gamma_{n,N^{(\Delta)}}(j)\right]=0$ for $j \neq 0$, we obtain for $j>0$
\begin{eqnarray} \label{star}
    n\E[\|\overline\Gamma_{n,N^{(\Delta)}}(j)-\E[\overline\Gamma_{n,N^{(\Delta)}}(j)]\|^2]
&=&\frac{n-j}{n}\sum_{S,T=1}^{r}\E\left[\left(N_{1+j}^{(\Delta)}N_{1}^{(\Delta)\top}\right)[S,T]^2\right] \nonumber\\
&\leq &\sum_{S,T=1}^{r}\Var\left(\left(N_{1}^{(\Delta)}N_{2}^{(\Delta)\top}\right)[S,T]\right),
\end{eqnarray}
and with similar calculations we obtain the same bound for $j<0.$

\item[(b)] Due to Hölder inequality we receive
\beao
    \sup_{g\in \mathcal G_r}\|\widetilde E_{n,N^{(\Delta)}}^{(M)}(g)\|
        &\leq&\sqrt n \sup_{g\in \mathcal G_r}\sqrt{\sum_{h=-M}^{M}\|\widehat g_h\|^2}\sqrt{\sum_{h=-M}^{M}\|\overline \Gamma_{n,N^{(\Delta)}}(h)-\E[\overline\Gamma_{n,N^{(\Delta)}}(h)]\|^2}\\
        &\leq& K\sqrt{n {\sum_{h=-M}^{M}\|\overline \Gamma_{n,N^{(\Delta)}}(h)-\E[\overline\Gamma_{n,N^{(\Delta)}}(h)]\|^2}}.
\eeao
Thus, an application of \eqref{star} gives
\beao
     \E\left[\sup_{g\in \mathcal G_r}\|\widetilde E_{n,N^{(\Delta)}}^{(M)}(g)\|^2\right] 
        &\leq& 2K^2 Mn\max_{h=-M,\ldots,M}\E\|\overline \Gamma_{n,N^{(\Delta)}}(h)-\E[\overline\Gamma_{n,N^{(\Delta)}}(h)]\|^2<\infty,
\eeao
the statement.
\end{itemize}
\end{proof}

 \noindent\textit{Proof of \Cref{aux2}}~\\
  By \Cref{aux3} there exists a constant $K>0$ such that
 \beam \label{B5}
 \E\left[\sup_{g\in \mathcal G_r}\|\widetilde E_{n,N^{(\Delta)}}^{(M)}(g)\|^2\right]\leq KM \quad \forall\, M\in \N.
  \eeam
  Let $\varepsilon>0$. Then, due to \eqref{unibound} there exists an $M_0\in \N$ such that
 \beam \label{B4}
    \limsup_{n\to\infty}\P\left(\sup_{g\in \mathcal G_r}\sqrt{n}\left\| \sum_{M<|h|\leq n}\widehat g_h \overline \Gamma_{n,N^{(\Delta)}}(h)\right\|>\sqrt{\frac{2K}{\varepsilon}}\right)\leq \frac{\varepsilon}{2} \quad \forall\, M\geq M_0.
 \eeam
 Define $\delta:=\sqrt{8KM_0/\varepsilon}\geq \sqrt{ 8K/\varepsilon}$. Hence, Markov's inequality and \eqref{B5} result in
 \beao
    \lefteqn{\P\left(\sup_{g\in \mathcal G_r}\left\|E_{n,N^{(\Delta)}}(g)\right\|>\delta\right)}\\
    &&\leq\P\left(\sup_{g\in \mathcal G_r}\sqrt{n}\left\| \sum_{M_0<|h|\leq n}\widehat g_h \overline \Gamma_{n,N^{(\Delta)}}(h)\right\|>\delta/2\right)
    +\P\left(\sup_{g\in \mathcal G_r}\left\|\widetilde E_{n,N^{(\Delta)}}^{(M_0)}(g)\right\|>\delta/2\right)\\
    &&\leq \P\left(\sup_{g\in \mathcal G_r}\sqrt{n}\left\| \sum_{M_0<|h|\leq n}\widehat g_h \overline \Gamma_{n,N^{(\Delta)}}(h)\right\|>\delta/2\right)
    +\frac{4}{\delta^2} KM_0. 
\eeao
Finally, an application of \eqref{B4} yields
\beao
    \limsup_{n\to\infty}\P\left(\sup_{g\in \mathcal G_r}\left\|E_{n,N^{(\Delta)}}(g)\right\|>\delta\right)
    \leq \frac{\varepsilon}{2}+\frac{\varepsilon}{2},
\eeao
and hence, the statement follows.
\hfill$\Box$

\end{appendix}

\end{document}